\documentclass[a4paper,12pt]{article}

\usepackage[utf8]{inputenc}
\usepackage[english]{babel}
\usepackage{enumitem}
\usepackage{float}
\usepackage{amsmath}
\usepackage{amssymb}
\usepackage{amsthm}
\usepackage{amsfonts}
\usepackage{booktabs}
\usepackage{epsfig}
\usepackage{epstopdf}
\usepackage{cite}
\usepackage{rotating}
\usepackage{algorithm,algpseudocode}

\usepackage{caption}
\newtheorem{dfn}{Definition}[section]
\newtheorem{lem}[dfn]{Lemma}
\newtheorem{thm}[dfn]{Theorem}

\theoremstyle{definition}

\newtheorem{asm}[dfn]{Assumption}
\newtheorem{exm}{Example}[section]
\newtheorem{rem}[dfn]{Remark}

\allowdisplaybreaks

\title{Weak Convergence for Variational Inequalities with Inertial-Type Method}

\date{\today}

\author{Yekini
Shehu\footnote{Department of Mathematics, Zhejiang Normal University, Jinhua,
321004, People's Republic of China; Institute of Science
and Technology (IST), Am Campus 1, 3400, Klosterneuburg, Austria; e-mail: yekini.shehu@unn.edu.ng. }
\hspace*{0.8mm}
Olaniyi. S. Iyiola\footnote{Department of Mathematics, Computer Science and Information Systems, California University of Pennsylvania, PA, USA; e-mail:niyi4oau@gmail.com.}}

\begin{document}

\maketitle

\begin{abstract}
\noindent Weak convergence of inertial iterative method for solving variational inequalities is the focus of this paper. The cost function is assumed to be non-Lipschitz and monotone. We propose a projection-type method with inertial terms and give weak convergence analysis under appropriate conditions. Some test results are performed and compared with relevant methods in the literature to show the efficiency and advantages given by our proposed methods.
\end{abstract}

\section{Introduction}\label{Sec:Intro}

\noindent
Suppose $C$ is a nonempty, closed and convex subset of a real Hilbert space $H$ and $F:C\rightarrow H$ a continuous mapping. The variational inequality problem (for short,
VI($F,C$)) is defined as: find $x\in C$ such that
\begin{eqnarray}\label{bami2}
   \langle F(x), y-x\rangle \geq 0, \quad \forall y \in C.
\end{eqnarray}
We shall denote by $ \text{SOL} $ the solution set of VI($F,C$) in \eqref{bami2}. Various applications of variational inequality can be found in \cite{Aubin,Baiocchi,Fichera,Fichera2,Glowinski,Khobotov,Kinderlehrer,Konnov,Marcotte}. \\

\noindent Projection-type method for solving VI($F,C$) \eqref{bami2} have been considered severally in the literature (see, for example, \cite{Ceng,CengYao1,Censor1,Dong,Hieu,MaingePE1,Malitsky3,Malitsky,Nadezhkina,Popov,Tseng}). Several other related methods to extragradient method and \eqref{pppp3} for solving VI($F,C$) \eqref{bami2} in real Hilbert spaces when $F$ is monotone and $L$-Lipschitz-continuous mapping have been studied in the literature (see, for example, \cite{Ceng,CengYao1,Censor1,Dong,Hieu,Korpelevich,MaingePE1,Malitsky3,Malitsky,Nadezhkina,Tseng}). Some of these methods involve  computing projection onto the feasible set $C$ twice per iteration and this can affect the efficiency of the methods.\\

\noindent
In \cite{Censor2}, Censor et al. introduced the subgradient extragradient method: $x_1\in H$,
\begin{eqnarray}\label{pppp3}
\left\{  \begin{array}{llll}
         & y_n = P_C(x_n-\lambda F(x_n)),\\
         & T_n:=\{w \in H:\langle x_n-\lambda F(x_n)-y_n,w-y_n\rangle \leq 0\},\\
         & x_{n+1}=P_{T_n}(x_n-\lambda F(y_n))
         \end{array}
         \right.
\end{eqnarray}
and gave weak convergence result when $F$ is monotone and $L$-Lipschitz-continuous mapping where $\lambda \in (0,\frac{1}{L})$. \\

\noindent
In order to accelerate the convergence of subgradient extragradient method \eqref{pppp3} and using the idea of in
\cite{Alvarez,Attouch,Attouch2,Attouch3,AttouchPey,Beck,Bot2,Bot3,CChen,Lorenz,Mainge2,Ochs,Polyak2},
Thong and Hieu \cite{Thong} introduced the following inertial subgradient extragradient method: $x_0, x_1\in H$,
\begin{eqnarray}\label{mm2}
\left\{  \begin{array}{llll}
         & w_n=x_n+\alpha_n(x_n-x_{n-1}),\\
         & y_n = P_C(w_n-\lambda F(w_n)),\\
         & T_n:=\{w \in H:\langle w_n-\lambda F(w_n)-y_n,w-y_n\rangle \leq 0\},\\
         & x_{n+1}=P_{T_n}(w_n-\lambda F(y_n))
         \end{array}
         \right.
\end{eqnarray}
and proved that $\{x_n\}$ generated by \eqref{mm2} converges weakly
to a solution of VI($F,C$) \eqref{bami2} when $F$ is monotone and $L$-Lipschitz-continuous mapping $ F $ where $0<
\lambda L\leq \frac{\frac{1}{2}-2\alpha-\frac{1}{2}\alpha^2-\delta}{\frac{1}{2}-\alpha+\frac{1}{2}\alpha^2}$ for some
$0<\delta<\frac{1}{2}-2\alpha-\frac{1}{2}\alpha^2$ and $\{\alpha_n\}$ is a non-decreasing sequence with $0\leq \alpha_n \leq \alpha<\sqrt{5}-2$.\\

\noindent The step-sizes in above methods \eqref{pppp3} and \eqref{mm2} are bounded by the inverse of the Lipschitz constant and this is quite inefficient, since in most applications a global Lipschitz constant (if it indeed exists at all) of $F$ cannot be accurately estimated, and is usually overestimated. This leads to too small step-sizes, which, of course, is not practical. Therefore, algorithms \eqref{pppp3} and \eqref{mm2} are not applicable in most cases of interest. This can be overcome by using an Armijo type line search procedure (see \cite{Khobotov,Marcotte,Solodov}).\\

\noindent  We provide a simple example of a variational inequality problem where the method \eqref{pppp3}
 proposed in \cite{Censor2} and method \eqref{mm2} proposed in \cite{Thong} cannot be applied.

\begin{exm}\label{aje}
Suppose $F:[0,\infty)\rightarrow \mathbb{R}$ is defined by $F(x):=e^x,~x \in [0,\infty)$. It is easy to see that $F$ is not Lipschitz continuous on $[0,\infty)$. By the mean value theorem, one has for an arbitrary $r>0$,
$$
|F(x)-F(y)|\leq e^r|x-y|
$$
 \noindent with $|x|,|y|\leq r$.
Hence, $F$ is uniformly continuous on bounded subsets of $C:=[0,\infty)$. Consequently, one can easily see that $F$ is monotone on $[0,\infty)$
since
$$
\langle F(x)-F(y),x-y\rangle=(F(x)-F(y))(x-y)\geq 0,~~\forall x,y \in [0,\infty).
$$
 \noindent Finally, $\text{SOL}$ of VI$(F,C)$ is nonempty since $0 \in \text{SOL}$.

\end{exm}
\noindent Motivated by Example \ref{aje}, it would be of interest to propose an iterative method for solving VI($F,C$) \eqref{bami2} for which the underline cost function $F$ is uniformly continuous on bounded subsets of $C$ but not Lipschitz continuous on $C$.\\

\noindent
Our interest in this paper is to obtain weak convergence results using inertial projection-type algorithm for VI($F,C$) \eqref{bami2} when the underline operator $F$ is monotone and uniformly continuous. We do not assume the cost function to be Lipschitz continuous as assumed in \cite{Censor1,Censor2,Dong,MaingePE1,Malitsky3,Thong}. Our proposed method is much more practical and outperforms the methods \eqref{pppp3} and \eqref{mm2} numerically.\\

\noindent
We organize the paper as follows: Basic
definitions and results are given in Section~\ref{Sec:Prelims} and the proposed method is introduced in
Section~\ref{Sec:Method}. We give weak convergence analysis of
the proposed method in Section~\ref{Sec:Convergence} and give some numerical comparisons of our method with methods \eqref{pppp3} and \eqref{mm2} in Section~\ref{Sec:Numerics}. Finally, we some concluding remarks in Section~\ref{Sec:Final}.

\section{Preliminaries}\label{Sec:Prelims}
\noindent
Suppose we take $ H $ as a real Hilbert space and $ X \subseteq H $ be a nonempty subset.
\begin{dfn}\label{Def:LipMon}
A mapping $ F: X \to H $
is called
\begin{itemize}
\item[{\rm(a)}] {\em monotone} on $X$ if $\langle
F(x) - F(y), x-y\rangle\geq0$ for all $x,y\in X$;
\item[{\rm(b)}] {\em Lipschitz continuous} on $X$ if there exists a
      constant $L>0$ such that
      $$\|F(x) - F(y)\|\leq L\|x-y\|,\ \forall x,y\in X.$$
\item[{\rm(c)}] {\em sequentially weakly continuous} if for each sequence $\{x_n\}$ we have:  $\{x_n\}$ converges weakly to $x$ implies
$\{F(x_n)\}$ converges weakly to $F(x)$.

\end{itemize}
\end{dfn}

\noindent Given any point $u \in H$, there exists a unique point $P_C u \in C$ (see, e.g., \cite{Bauschkebook}) such that
$$\|u-P_C u\|\leq\|u-y\|,~\forall y \in C.$$
This $P_C$ is called the {\it metric projection } of $H$ onto $C$. It is known that $P_C$ is a nonexpansive mapping of $H$ onto $C$ and satisfies
\begin{equation}\label{a21}
   \langle x-y, P_C x-P_C y \rangle \geq \|P_C x-P_C y\|^2,~~\forall x, y \in H.
\end{equation}
In particular, we get from \eqref{a21} that
\begin{equation}\label{load}
   \langle x-y, x-P_C y \rangle \geq \|x-P_C y\|^2,~~\forall x \in C, y \in H.
\end{equation}
Another property of $P_C x$ is :
\begin{equation}\label{a22}
   P_Cx\in C\quad \text{and} \quad \langle x-P_C x,P_C x-y\rangle\geq0,~\forall y\in C.
\end{equation}
More details on $P_C$ can be found, for example, in Section 3 of \cite{Goebel}.\\

\noindent
The following results are needed in the next section.

\begin{lem}\label{lm2}
The following statements hold in $ H $:
\begin{itemize}
   \item[(a)] $ \|x+y\|^2=\|x\|^2+2\langle x,y\rangle+\|y\|^2 $ for all
      $ x, y \in H $;
   \item[(b)] $ 2 \langle x-y, x-z \rangle = \| x-y \|^2 + \| x-z \|^2 -
      \| y-z \|^2 $ for all $ x,y,z \in H $;
   \item[(c)] $\|\alpha x+(1-\alpha)y\|^2 =\alpha \|x\|^2+(1-\alpha)\|y\|^2
   -\alpha(1-\alpha)\|x-y\|^2$ for all $x,y \in H$ and $\alpha \in \mathbb{R}$.

\end{itemize}
\end{lem}

\begin{lem}\label{la5}(see \cite[Lem.\ 3]{alvarez}) Let $\{ \psi_n\}$, $\{ \delta_n\}$ and $\{ \alpha_n\}$ be the sequences in $[0,+\infty)$ such that $\psi_{n+1}\leq \psi_n+\alpha_n(\psi_n-\psi_{n-1})+\delta_n$ for all $n\geq1$, $\sum_{n=1}^{\infty}\delta_n<+\infty$ and there exists a real number $\alpha$ with $0\leq\alpha_n\leq\alpha<1$ for all $n\geq1$. Then the following hold:\\
$(i)~~\sum_{n\geq1}[\psi_n-\psi_{n-1}]_+<+\infty$, where $[t]_+=\max\{ t,0\}$;\\
(ii) there exists $\psi^*\in[0,+\infty)$ such that $\lim_{n\rightarrow+\infty}\psi_n=\psi^*$.
\end{lem}

\begin{lem}\label{super}(see \cite[Lem.\ 2.39]{Bauschkebook})
Let $C$ be a nonempty set of $H$ and $\{ x_n\}$ be a sequence in $H$ such that the following two conditions hold:\\
(i) for any $x\in C$, $\lim_{n\rightarrow\infty}\| x_n-x\|$ exists;\\
(ii) every sequential weak cluster point of $\{ x_n\}$ is in $C$.\\
Then $\{ x_n\}$ converges weakly to a point in $C$.
\end{lem}

\begin{lem}\label{new1} (\cite{He2})
Let $C$ be a nonempty closed and convex subset of $H$. Let $h$ be a real-valued function on $H$ and define
$K:=\{x: h(x)\leq 0\}$. If $K$ is nonempty and $h$ is
Lipschitz continuous on $C$ with modulus $\theta>0$, then
$${\rm dist}(x,K) \geq\theta^{-1}\max \{h(x),0\},~\forall x \in C,$$
\noindent where ${\rm dist}(x,K)$ denotes the distance function from
$x$ to $K$.
\end{lem}

\begin{lem}\label{new3}
Let $C$ be a nonempty closed and convex subset of $H$, $y:=P_C(x)$ and $x^* \in C$. Then
\begin{equation}\label{SpCseEst}
    \|y-x^*\|^2\leq\|x-x^*\|^2-\|x-y\|^2.
\end{equation}
\end{lem}

\begin{lem}\label{lm28a} (\cite[Prop.\ 2.11]{Iusem1}, \cite[Prop.\ 4]{Iusem3})
Let $H_1$ and $H_2$ be two real Hilbert spaces. Suppose
$F:H_1\rightarrow H_2$ is uniformly continuous on bounded subsets of $H_1$
and $M$ is a bounded subset of $H_1$. Then $F(M)$ is bounded.
\end{lem}

\begin{lem}\label{lm27a}(\cite[Lem.\ 7.1.7]{wtak})
Let $C$ be a nonempty, closed, and convex subset of $H$.
Let $F:C\rightarrow H$ be a continuous, monotone mapping and $z \in C$. Then
$$
   z \in {\rm SOL} \Longleftrightarrow \langle F(x), x-z\rangle
   \geq 0 \quad \text{for all } x\in C.
$$
\end{lem}

\section{Proposed Method}\label{Sec:Method}

\noindent
We give some assumptions on the feasible set $C$, the cost function $F$ and the iterative parameter $\{ \alpha_n \} $ below.

\begin{asm}\label{Ass:VI}
Suppose that the following hold:
\begin{itemize}
\item[{\rm(a)}] The feasible set $ C $ is a nonempty closed affine
      subset of the real Hilbert space $ H $.
\item[{\rm(b)}] $ F: C \to H $ is monotone and  uniformly continuous on bounded subsets of $H$.
\item[{\rm(c)}] The solution set $\text{SOL}$ of VI$(F,C)$ is nonempty.
\end{itemize}
\end{asm}

\begin{asm}\label{Ass:Parameters}
Suppose the real sequence $ \{ \alpha_n \} $
satisfy the following condition:
\begin{itemize}
  \item $ \{ \alpha_n \} \subset (0,1) $
       with  $0\leq \alpha_n \leq \alpha_{n+1} \leq \alpha<\frac{1}{3}$ for all $n$.
\end{itemize}
\end{asm}

\noindent
Suppose we define
$$
   r(x) := x - P_C (x-F(x))
$$
as the residual equation. Then if $y=x-F(x)$ in \eqref{load}, we obtain
\begin{equation}\label{load2}
   \langle F(x),r(x)\rangle\geq\|r(x)\|^2,~\forall x\in C.
\end{equation}

\noindent We next give our proposed inertial projection-type method.

\begin{algorithm}[H]
\caption{Inertial Projection Method }\label{Alg:AlgL}
\begin{algorithmic}[1]
\State Choose sequence $ \{ \alpha_n \} $ and $\sigma \in (0,1)$ such that the conditions from Assumption~\ref{Ass:Parameters} hold,
      and take $\gamma \in (0,1)$.
      Let $ x_0= x_1 \in H $ be a given starting point. Set $ n := 1 $.
\State Set
       \begin{equation*}
         w_n := x_n+\alpha_n (x_n-x_{n-1}).
       \end{equation*}
Compute $z_n:=P_C(w_n-F(w_n))$. If $r(w_n)=w_n-z_n=0$: STOP.
\State Compute $y_n=w_n-\gamma^{m_n}r(w_n)$, where $m_n$ is the smallest nonnegative integer satisfying
         \begin{equation}\label{ee31}
         \langle F(y_n),r(w_n)\rangle\geq\frac{\sigma}{2}\|r(w_n)\|^2.
         \end{equation}
Set $\eta_n:=\gamma^{m_n}$.
\State Compute
         \begin{equation}\label{e31}
         x_{n+1}=P_{C_n}(w_n),
         \end{equation}
where $C_n=\{x: h_n(x)\leq0\}$ and
         \begin{equation}\label{ego}
         h_n(x):=\langle F(y_n),x-y_n\rangle.
         \end{equation}
\State Set $n\leftarrow n+1$ and \textbf{goto 2}.
\end{algorithmic}
\end{algorithm}
\noindent
If $ r(w_n) = 0 $, then $w_n$ is a solution of VI($F,C$) \eqref{bami2}. In the analysis we assume that $ r(w_n) \neq 0 $ for infinitely many iterations, so that
Algorithm~\ref{Alg:AlgL} generates an infinite sequence satisfying $ r(w_n) \neq 0 $ for all $ n \in \mathbb{N} $.

\begin{rem}\label{Rem:Simple}
(a) Our proposed Algorithm~\ref{Alg:AlgL} requires, at each iteration, only one
projection onto the feasible set $C $ and another projection onto the half-space $C_n$ (which has a closed form solution, \cite{Cegielskibook}) and this is numerically less expensive than the twice computation of projection onto $C$ per iteration in extragradient method \cite{Korpelevich}. \\[-1mm]

\noindent (b) As we have mentioned before, Algorithm~\ref{Alg:AlgL} is much more applicable than \eqref{pppp3} and \eqref{mm2} because the Lipschtz constant of the cost function $F$ is not needed during implementations.  \hfill $\Diamond$
\end{rem}

\begin{lem}\label{new2}
Let the function $h_n$ be defined by
\eqref{ego}. Then
$$h_n(w_n)\geq\frac{\sigma\eta_n}{2}\|w_n-z_n\|^2.$$
\noindent In particular, if $w_n\neq z_n$, then
$h_n(w_n)>0$. If $x^*\in \text{SOL}$, then $h_n(x^*)\leq 0$.
 \end{lem}

\begin{proof}
Since $y_n=w_n-\eta_n(w_n-z_n)$, using \eqref{ee31} we have
$$\begin{aligned}
    h_n(w_n)
    &=\langle F(y_n),w_n-y_n\rangle\\
    &=\eta_n\langle F(y_n),w_n-z_n\rangle\geq\eta_n\frac{\sigma}{2}\|w_n-z_n\|^2\geq0.
\end{aligned}$$
If $w_n\neq z_n$, then $h_n(w_n)>0$.
Furthermore, suppose $x^* \in \text{SOL}$. Then by Lemma \ref{lm27a} we have
$\langle F(x), x-x^*\rangle   \geq 0 \quad \text{for all } x\in C.$
In particular, $\langle F(y_n), y_n-x^*\rangle   \geq 0$ and hence $h_n(x^*) \leq 0.$
\end{proof}

\section{Convergence Analysis}\label{Sec:Convergence}
\noindent
Let us give weak convergence analysis of our proposed Algorithm~\ref{Alg:AlgL} in this section.

\begin{lem}\label{nece1}
Let $\{x_n\}$ be generated by Algorithm~\ref{Alg:AlgL}. Then under Assumptions~\ref{Ass:VI}
and \ref{Ass:Parameters}, we have that \\
(i)~~$\{x_n\}$ is bounded, and \\
(ii)~~$ \lim_{n\rightarrow \infty} \|x_{n+1}-w_n\|=0$.
\end{lem}

\begin{proof}
Let $x^*\in\text{SOL}$.
By Lemma~\ref{new3} we get (since $x^*\in
C_n$) that
\begin{eqnarray}\label{jj7}
    \|x_{n+1}-x^*\|^2&=&\|P_{C_n}(w_n)-x^*\|^2\leq\|w_n-x^*\|^2-\|x_{n+1}-w_n\|^2 \\
    &=&\|w_n-x^*\|^2-{\rm dist}^2(w_n,C_n)\nonumber.
\end{eqnarray}
Now, using Lemma \ref{lm2} (c), we have
\begin{eqnarray}\label{asu4}
    \|w_n-x^*\|^2&=&\|(1+\alpha_n)(x_n-x^*)-\alpha_n(x_{n-1}-x^*)\|^2\nonumber\\
    &=& (1+\alpha_n)\|x_n-x^*\|^2-\alpha_n\|x_{n-1}-x^*\|\nonumber\\
    &&+\alpha_n(1+\alpha_n)\|x_n-x_{n-1}\|^2.
    \end{eqnarray}
Also,
\begin{eqnarray}\label{e7}
\|x_{n+1}-w_n\|^2&=& \|x_{n+1}-(x_n+\alpha_n(x_n-x_{n-1}))\|^2 \nonumber \\
  &=& \|x_{n+1}-x_n\|^2+\alpha^2_k\|x_n-x_{n-1}\|^2-2\alpha_n \langle x_{n+1}-x_n,x_n-x_{n-1}\rangle \nonumber\\
  &\geq & \|x_{n+1}-x_n\|^2+\alpha^2_k\|x_n-x_{n-1}\|^2-2\alpha_n \|x_{n+1}-x_n\|\|x_n-x_{n-1}\| \nonumber\\
  &\geq & \|x_{n+1}-x_n\|^2+\alpha^2_k\|x_n-x_{n-1}\|^2-\alpha_n \|x_{n+1}-x_n\|^2\nonumber\\
  &&-\alpha_n \|x_n-x_{n-1}\|^2\nonumber\\
  &=&(1-\alpha_n)\|x_{n+1}-x_n\|^2+(\alpha_n^2-\alpha_n)\|x_n-x_{n-1}\|^2.
\end{eqnarray}
Combining \eqref{jj7}, \eqref{asu4} and \eqref{e7}, we get
\begin{eqnarray}\label{e8}
\|x_{n+1}-x^*\|^2&\leq & (1+\alpha_n)\|x_n-x^*\|^2-\alpha_n\|x_{n-1}-x^*\|^2 \nonumber \\
  &&+ \alpha_n(1+\alpha_n)\|x_n-x_{n-1}\|^2-(1-\alpha_n)\|x_{n+1}-x_n\|^2\nonumber \\
  &&-(\alpha_n^2-\alpha_n)\|x_n-x_{n-1}\|^2\nonumber \\
  &=&(1+\alpha_n)\|x_n-x^*\|^2-\alpha_n\|x_{n-1}-x^*\|^2 \nonumber \\
  &&-(1-\alpha_n)\|x_{n+1}-x_n\|^2+(\alpha_n(1+\alpha_n)-(\alpha_n^2-\alpha_n))\|x_n-x_{n-1}\|^2\nonumber \\
  &=&(1+\alpha_n)\|x_n-x^*\|^2-\alpha_n\|x_{n-1}-x^*\|^2 \nonumber \\
  &&-(1-\alpha_n)\|x_{n+1}-x_n\|^2+2\alpha_n\|x_n-x_{n-1}\|^2.
\end{eqnarray}
Using the fact that $\alpha_n\leq \alpha_{n+1}$, we obtain from \eqref{e8} that
\begin{eqnarray}\label{e9}
\|x_{n+1}-x^*\|^2  &\leq & (1+\alpha_{n+1})\|x_n-x^*\|^2-\alpha_n\|x_{n-1}-x^*\|^2 \nonumber \\
  &&-(1-\alpha_n)\|x_{n+1}-x_n\|^2+2\alpha_n\|x_n-x_{n-1}\|^2.
\end{eqnarray}
By \eqref{e9}, we get
\begin{eqnarray*}
&&\|x_{n+1}-x^*\|^2 -\alpha_{n+1}\|x_n-x^*\|^2+2\alpha_{n+1}\|x_{n+1}-x_n\|^2\leq
\|x_n-x^*\|^2-\alpha_n\|x_{n-1}-x^*\|^2\nonumber \\
&& +2\alpha_n\|x_n-x_{n-1}\|^2+2\alpha_{n+1}\|x_{n+1}-x_n\|^2-(1-\alpha_n)\|x_{n+1}-x_n\|^2\nonumber \\
&=& \|x_n-x^*\|^2-\alpha_n\|x_{n-1}-x^*\|^2+2\alpha_n\|x_n-x_{n-1}\|^2+(2\alpha_{n+1}-1+\alpha_n)\|x_{n+1}-x_n\|^2.
\end{eqnarray*}
Therefore,
\begin{eqnarray}\label{e5}
\|x_{n+1}-x^*\|^2 &\leq& \|x_n-x^*\|^2-\alpha_n\|x_{n-1}-x^*\|^2+2\alpha_n\|x_n-x_{n-1}\|^2\nonumber \\
&&+(2\alpha_{n+1}-1+\alpha_n)\|x_{n+1}-x_n\|^2.
\end{eqnarray}
\noindent Let us define
$$
\Gamma_n:= \|x_n-x^*\|^2-\alpha_n\|x_{n-1}-x^*\|^2+2\alpha_n\|x_n-x_{n-1}\|^2.
$$
\noindent  Then we have from \eqref{e5} that
\begin{equation}\label{e10}
\Gamma_{n+1}-\Gamma_n \leq (2\alpha_{n+1}-1+\alpha_n)\|x_{n+1}-x_n\|^2.
\end{equation}
Since  $0\leq \alpha_n\leq \alpha_{n+1}\leq \alpha<\frac{1}{3}$, we get
$-2\alpha_{n+1}\geq -2\alpha$ and $-\alpha_n \geq -\alpha$. This implies that
$-(2\alpha_{n+1}-1+\alpha_n)=-2\alpha_{n+1}+1-\alpha_n \geq -2\alpha +1-\alpha \geq 1-3\alpha>0$ since
$\alpha< \frac{1}{3}$. Now, let us define $\sigma:=1-3\alpha$. Then
\begin{equation}\label{e11}
 2\alpha_{n+1}-1+\alpha_n \leq -\sigma.
\end{equation}
Putting \eqref{e11} into \eqref{e10}, we have
\begin{equation}\label{e12}
\Gamma_{n+1}-\Gamma_n \leq -\sigma\|x_{n+1}-x_n\|^2.
\end{equation}
From \eqref{e12}, we see that $\{\Gamma_n\}$ is monotone nonincreasing. Furthermore,
\begin{eqnarray}\label{e13}
\Gamma_n&=& \|x_n-x^*\|^2-\alpha_n\|x_{n-1}-x^*\|^2+2\alpha_n\|x_n-x_{n-1}\|^2 \nonumber \\
  &\geq & \|x_n-x^*\|^2-\alpha_n\|x_{n-1}-x^*\|^2.
\end{eqnarray}
So,
\begin{eqnarray}\label{e14}
\|x_n-x^*\|^2&\leq & \alpha_n\|x_{n-1}-x^*\|^2+\Gamma_n \nonumber \\
   &\leq & \alpha\|x_{n-1}-x^*\|^2+\Gamma_1 \nonumber \\
   &\vdots&\nonumber \\
  &\leq & \alpha^k\|x_0-x^*\|^2+(1+\alpha+\alpha^2+\ldots+\alpha^{k-1})\Gamma_1 \nonumber\\
   &=& \alpha^k\|x_0-x^*\|^2+ \frac{\Gamma_1}{1-\alpha}.
\end{eqnarray}
From \eqref{e14}, we can infer that $\{x_n\}$ is bounded.
Using the definition of $\Gamma_n$, we have
\begin{eqnarray}\label{e15}
\Gamma_{n+1} &=& \|x_{n+1}-x^*\|^2-\alpha_{n+1}\|x_n-x^*\|^2+2\alpha_{n+1}\|x_{n+1}-x_n\|^2 \nonumber\\
   &\geq & -\alpha_{n+1}\|x_n-x^*\|^2.
\end{eqnarray}
Using \eqref{e14} in \eqref{e15}, we get

\begin{eqnarray}\label{e16}
-\Gamma_{n+1} &\leq & -\alpha_{n+1}\|x_n-x^*\|^2\leq \alpha\|x_n-x^*\|^2 \nonumber\\
&\leq & \alpha^{k+1}\|x_0-x^*\|^2+\frac{\alpha \Gamma_1}{1-\alpha}.
\end{eqnarray}
From \eqref{e12}, we get
$$
\sigma\|x_{n+1}-x_n\|^2\leq \Gamma_n-\Gamma_{n+1}
$$
\noindent and so
\begin{eqnarray*}
\sigma\sum_{j=1}^{n}\|x_{j+1}-x_j\|^2&\leq& \Gamma_1-\Gamma_{n+1}\\
&\leq& \Gamma_1+\alpha^{n+1}\|x_0-x^*\|^2+\frac{\alpha \Gamma_1}{1-\alpha}\\
&\leq& \alpha^{n+1}\|x_0-x^*\|^2+\frac{\Gamma_1}{1-\alpha}\\
&\leq& \|x_0-x^*\|^2+\frac{\Gamma_1}{1-\alpha}.
\end{eqnarray*}
Therefore, since $x_0=x_1$, we get
\begin{eqnarray*}
\sum_{k=1}^{\infty}\|x_{n+1}-x_n\|^2&\leq& \frac{1}{\sigma}\Big(\|x_0-x^*\|^2+\frac{\Gamma_1}{1-\alpha}\Big)\\
&=& \frac{1}{\sigma}\|x_0-x^*\|^2+\frac{1-\alpha_1}{1-\alpha}\|x_0-x^*\|^2\\
&=&\Big(\frac{1}{1-3\alpha}+\frac{1-\alpha_1}{1-\alpha}\Big)\|x_0-x^*\|^2<\infty.
\end{eqnarray*}
\noindent Observe that
\begin{eqnarray}\label{e17}
\|x_{n+1}-w_n\|&=&\|x_{n+1}-x_n-\alpha_n(x_n-x_{n-1})\|\nonumber \\
  &\leq& \|x_{n+1}-x_n\|+\alpha_n\|x_n-x_{n-1}\|\nonumber \\
  &\leq& \|x_{n+1}-x_n\|+\alpha\|x_n-x_{n-1}\|.
\end{eqnarray}
Using \eqref{e17}, we obtain
\begin{equation}\label{e18}
 \lim_{n\rightarrow \infty} \|x_{n+1}-w_n\|=0.
\end{equation}
\end{proof}

\begin{lem}\label{nece3}
Let $\{x_n\}$ generated by Algorithm~\ref{Alg:AlgL} above and
Assumptions~\ref{Ass:VI} and \ref{Ass:Parameters} hold.
Then
\begin{itemize}
\item[{\rm(a)}] $\displaystyle\lim_{n\rightarrow \infty} \eta_{n}\|w_n-z_n\|^2=0$;
\item[{\rm(b)}] $\displaystyle\lim_{n\rightarrow \infty} \|w_n-z_n\|=0.$
\end{itemize}
\end{lem}

\begin{proof}
Let $x^* \in \text{SOL} $.
Since $F$ is uniformly
continuous on bounded subsets of $X$, then $\{F(x_n)\}, \{z_n\},
\{w_n\}$ and $\{F(y_n)\}$ are bounded. In particular, there exists
$M>0$ such that $\|F(y_n)\|\leq M$ for all $n\in\mathbb{N}$. Combining
Lemma~\ref{new1} and Lemma~\ref{new2}, we get
\begin{eqnarray}\label{jjj7}
    \|x_{n+1}-x^*\|^2&=&\|P_{C_n}(w_n)-x^*\|^2\leq\|w_n-x^*\|^2-\|x_{n+1}-w_n\|^2 \nonumber \\
    &=&\|w_n-x^*\|^2-{\rm dist}^2(w_n,C_n)\nonumber\\
&\leq& \|w_n-x^*\|^2-\Big(\frac{1}{M}h_n(w_n)\Big)^2 \nonumber\\
    &\leq& \|w_n-x^*\|^2-\Big(\frac{1}{2M}\sigma\eta_n\|r(w_n)\|^2\Big)^2\nonumber\\
    &=&\|w_n-x^*\|^2-\Big(\frac{1}{2M}\sigma\eta_n\|w_n-z_n\|^2\Big)^2.
\end{eqnarray}
Since $\{x_n\}$ is bounded, we obtain from \eqref{jjj7} that
\begin{eqnarray}\label{asu13}
\Big(\frac{1}{2M}\sigma\eta_n\|w_n-z_n\|^2\Big)^2 &\leq& \|w_n-x^*\|^2- \|x_{n+1}-x^*\|^2\nonumber \\
&=&\Big(\|w_n-x^*\|- \|x_{n+1}-x^*\|\Big)\Big(\|w_n-x^*\|+ \|x_{n+1}-x^*\|\Big)\nonumber \\
&\leq& \|w_n-x^*\|- \|x_{n+1}-x^*\|M_1 \nonumber \\
&\leq& \|w_n-x_{n+1}\|M_1,
\end{eqnarray}
where $M_1:=\sup_{n\geq 1}\{\|w_n-x^*\|+ \|x_{n+1}-x^*\|\}$. This establishes (a).\\

\noindent To establish (b), We distinguish two cases depending on the behaviour of (the bounded)
sequence of step-sizes $\{\eta_n\}$.

\noindent \textbf{Case 1}: Suppose that $ \liminf_{n \to \infty}
\eta_n > 0 $. Then
$$0\leq \|r(w_n)\|^2=\frac{\eta_{n}\|r(w_n)\|^2}{\eta_n}$$
\noindent and this implies that
$$\begin{aligned}
    \limsup_{n\to\infty}\|r(w_n)\|^2
    &\leq\limsup_{n\to\infty}\bigg(\eta_n\|r(w_n)\|^2\bigg)
                             \bigg(\limsup_{n\to\infty} \frac{1}{\eta_n}\bigg)\\
    &=\bigg(\limsup_{n\to\infty}\eta_n\|r(w_n)\|^2\bigg)\frac{1}{\liminf_{n\to\infty}\eta_n}\\
    &=0.
\end{aligned}$$
Hence, $\limsup_{n \to \infty}\|r(w_n)\|=0$. Therefore,
$$\lim_{n\rightarrow\infty}\|w_n-z_n\|=\lim_{n\rightarrow\infty}\|r(w_n)\|=0.$$

\noindent \textbf{Case 2}: Suppose that $ \liminf_{n \to \infty}
\eta_n= 0 $. Subsequencing if necessary, we may assume without loss of generality that
$\lim_{n\to\infty}\eta_n=0$ and $\lim_{n\rightarrow \infty}\|w_n-z_n\|=a \geq 0$.

\noindent Define
$\bar{y}_n:=\frac{1}{\gamma}\eta_nz_n+\Big(1-\frac{1}{\gamma}\eta_n\Big)w_n$
or, equivalently,
$\bar{y}_n-w_n=\frac{1}{\gamma}\eta_n(z_n-w_n)$.
Since $\{z_n-w_n\}$ is bounded and since $\lim_{n \to
\infty}\eta_n=0$ holds, it follows that
\begin{equation}\label{ify2}
   \lim_{n\to\infty}\|\bar{y}_n-w_n\|=0.
\end{equation}
From the step-size rule and the definition of $\bar{y}_k$, we have
$$\langle F(\bar{y}_n),w_n-z_n\rangle<\frac{\sigma}{2}\|w_n-z_n\|^2,\ \forall n\in\mathbb{N},$$
or equivalently
$$
2\langle F(w_n),w_n-z_n\rangle+
2\langle F(\bar{y}_n)-F(w_n),w_n-z_n\rangle
<\sigma\|w_n-z_n\|^2,\ \forall n\in\mathbb{N}.$$
Setting $t_n:=w_n-F(w_n)$, we obtain form the last inequality that
$$
2\langle w_n-t_n,w_n-z_n\rangle
+
2\langle F(\bar{y}_n)-F(w_n),w_n-z_n\rangle
<\sigma\|w_n-z_n\|^2,\ \forall n\in\mathbb{N}.
$$
Using Lemma~\ref{lm2} (b)  we get
$$
2\langle w_n-t_n,w_n-z_n\rangle=\|w_n-z_n\|^2+\|w_n-t_n\|^2-\|z_n-t_n\|^2.
$$
Therefore,
$$
\|w_n-t_n\|^2-\|z_n-t_n\|^2 < (\sigma-1)\|w_n-z_n\|^2-2\langle F(\bar{y}_n)-F(w_n),w_n-z_n\rangle\ \forall n\in\mathbb{N}.
$$
Since $F$ is uniformly continuous on bounded subsets of $H$ and
\eqref{ify2}, if $a>0$ then the right hand side of the last inequality converges to $(\sigma-1)a<0$ as $n \to \infty$.
From the last inequality we have
$$
\limsup_{n \to \infty} \left( \|w_n-t_n\|^2-\|z_n-t_n\|^2 \right) \leq (\sigma-1)a < 0.
$$
For $\epsilon=-(\sigma-1)a/2 >0$, there
exists $N\in\mathbb{N}$ such that
$$\|w_n-t_n\|^2 - \|z_n-t_n\|^2 \leq (\sigma-1)a+\epsilon= (\sigma-1)a/2 <0  \quad \forall n\in \mathbb{N},n\geq N,$$
leading to
$$\|w_n-t_n\| < \|z_n-t_n\| \quad \forall n\in \mathbb{N},n\geq N, $$
which is a contradiction to the definition of
$z_n=P_C(w_n-F(w_n))$. Hence $a=0$,
which completes the proof.
\end{proof}

\begin{lem}\label{Lem:help}
Let Assumptions~\ref{Ass:VI} and \ref{Ass:Parameters} hold.
Furthermore let $\{x_{n_k}\}$ be a subsequence of $ \{x_n\}$
converging weakly to a limit point $ p $. Then $ p \in \text{SOL} $.
\end{lem}

\begin{proof}
By the definition of $z_{n_k}$ together with \eqref{a22}, we have
$$\langle w_{n_k}-F(w_{n_k})-z_{n_k},x-z_{n_k}\rangle\leq 0,\ \forall x\in C,$$
\noindent which implies that
$$\langle w_{n_k}-z_{n_k},x-z_{n_k}\rangle\leq\langle F(w_{n_k}),x-z_{n_k}
  \rangle,\ \forall x \in C.$$
Hence,
\begin{equation}\label{j11}
   \langle w_{n_k}-z_{n_k},x-z_{n_k}\rangle +\langle F(w_{n_k}),z_{n_k}-w_{n_k}\rangle
   \leq \langle F(w_{n_k}),x-w_{n_k}\rangle,\ \forall x\in C.
\end{equation}
Fix $x \in C$ and let $k\rightarrow \infty$ in \eqref{j11}.
Since
$\lim_{k \to \infty} \|w_{n_k}-z_{n_k}\|=0 $, we have
\begin{equation}\label{anbi}
    0\leq \liminf_{k \to \infty} \langle F(w_{n_k}),x-w_{n_k}\rangle
\end{equation}
for all $ x \in C $.
It follows from \eqref{j11} and the monotonicity of $F$ that
 \begin{eqnarray*}
 \langle w_{n_k}-z_{n_k},x-z_{n_k}\rangle +\langle F(w_{n_k}),z_{n_k}-w_{n_k}\rangle
    &\leq& \langle F(w_{n_k}),x-w_{n_k}\rangle\\
    &\leq& \langle F(x),x-w_{n_k}\rangle  \quad \forall x\in C.
 \end{eqnarray*}
Letting $k \to +\infty$ in the last inequality, remembering that  $\lim_{k \to \infty}\|w_{n_k}-z_{n_k}\|=0$  for all $k$, we have
$$
\langle F(x),x-p\rangle \geq 0 \quad  \forall x\in C.
$$
\noindent
In view of Lemma~\ref{lm27a}, this implies $p\in\text{SOL}$.
\end{proof}

\begin{thm}\label{t32}
Let Assumptions~\ref{Ass:VI} and \ref{Ass:Parameters} hold.
Then the sequence $ \{x_n\} $ generated by Algorithm~\ref{Alg:AlgL}
weakly converges to a point in $\text{SOL} $.
\end{thm}

\begin{proof}
We have shown that \\
(i) $\lim_{n \to \infty} \|x_n-x^*\|$ exists;\\
(ii) $\omega_w(x_n) \subset \text{SOL}$, where
$\omega_w(x_n):=\{x:\exists x_{n_j}\rightharpoonup x\}$ denotes the weak $\omega$-limit set of $\{x_n\}$.\\
Then, by Lemma \ref{super}, we have that $\{x_n\}$ converges weakly to a point in $\text{SOL} $.
\end{proof}

\begin{rem}
(a) One can still obtain weak convergence for Algorithm~\ref{Alg:AlgL} when $C$ is a nonempty, closed and convex subset of $H$. \\[-1mm]

\noindent
(b) In finite-dimensional spaces, Theorem \ref{t32} holds when $ F $ is monotone and continuous. \\[-1mm]

(c) Lemmas 3.5, 4.1, 4.2 and Theorem 4.4 can be obtained when $F$ pseudo-monotone and weakly sequentially continuous (i.e., for all $x,y \in H$,
$\langle F(x),y-x\rangle\geq0\Longrightarrow\langle F(y),y-x\rangle\geq0;$). The reader can see, for example, \cite{ShehuCOAM}.\hfill $\Diamond$
\end{rem}

\begin{rem}
Our proposed method in this paper gives weak convergence results in infinite dimensional Hilbert space. There exists strong convergence methods in the literature for solving variational inequality problem in infinite dimensional Hilbert space (see, for example, \cite{Ceng,Censor1,KanzowShehu2,MaingeSIAM,Malitsky,Mashreghi,Nadezhkina,ShehuCalcolo}). These methods use ideas of viscosity terms, Halpern iterations and hybrid methods. It has been shown numerically in \cite{KanzowShehu2} that viscosity and Halpern-type strongly convergent methods outperform those of hybrid methods. Nonetheless, proposed viscosity and Halpern-type strongly convergent methods involve the iterative parameter that is both diminishing and non-summable. These conditions on the iterative parameters make the viscosity and Halpern-type strongly convergent methods to be slower than our proposed method in this paper in terms number of iterations and CPU time.
\end{rem}

\section{Numerical Experiments}\label{Sec:Numerics}

In this section, we discuss the numerical behaviour of Algorithm~\ref{Alg:AlgL} using different test examples taken from the literature which are describe below and compare our method with \eqref{pppp3}, \eqref{mm2} and Shehu and Iyiola algorithm 3.2 in \cite{ShehuVIM}.\\

\begin{exm}\label{her1}{\bf Equilibrium-optimization Model}
\noindent\\
In this example, we consider an equilibrium-optimization model (see, for example, \cite{ShehuJIMO}) which can be regarded as an extension of a Nash-Cournot oligopolistic equilibrium model in electricity markets.\\

\noindent
In this equilibrium model, we assume that there are $m$ companies, each company $i$ may possess $I_i$ generating units.
Suppose we denote by $x$, the vector whose entry $x_j$ stands for the power generating by unit $j$.
Suppose the price $p_i(s)$ is a decreasing
affine function of $s$ where $s:=\sum_{j=1}^N x_j$ where $N$ is the number of all generating
units. Thus, $p_i(s):=\alpha-\beta_is$. Then the profit made by company $i$ is given by
$f_i(x):=p_i(s)\sum_{j \in I_i}x_j-\sum_{j \in I_i}c_j(x_j)$, where $c_j(x_j)$ is the cost for generating
$x_j$ by generating unit $j$ . Let us assume that $K_i$ is the strategy set of company $i$, which implies that
$\sum_{j \in I_i}x_j \in K_i$ for each $i$. Then the strategy set of the
model is $C:=K_1 \times K_2 \times\ldots \times K_m$.\\

\noindent
A commonly used approach when each company wants to maximize its profit by choosing the corresponding
production level under the presumption that the production of the other companies are
parametric input is the Nash equilibrium concept.\\

\noindent
We recall that a point $x^* \in C=K_1 \times K_2 \times\ldots \times K_m$ is an equilibrium point if
$$
f_i(x^*)\geq f_i(x^*[x_i]) \forall x_i \in K_i,~~i=1,2,\ldots,m,
$$
\noindent
where the vector $x^*[x_i]$ stands for the vector obtained from $x^*$ by replacing $x^*_i$
with $x_i$. Define
$$
f(x,y):=\psi(x,y)-\psi(x,x)
$$
\noindent with
$$
\psi(x,y):=-\sum_{i=1}^n f_i(x^*[y_i]).
$$
\noindent
Then the problem of finding a Nash equilibrium point of our model can be formulated as
\begin{eqnarray}\label{qe}
X^* \in C:f(x^*,x)\geq 0~~\forall x \in C.
\end{eqnarray}
\noindent
Suppose for every $j$, the cost $c_j$ for production and the environmental fee $g$ are increasingly convex functions. The convexity assumption here means that both the cost and fee for producing a unit production increases as the quantity of the production gets larger.
Under this convexity assumption, it is not hard to see that \eqref{qe} is equivalent to (see, \cite{Yen})
\begin{eqnarray}\label{qe1}
x \in C:\langle Bx-a+\nabla \varphi(x), y-x\rangle\geq 0~~\forall y \in C,
\end{eqnarray}
where
\begin{eqnarray*}
&&a:=(\alpha,\alpha,\ldots,\alpha)^T\\
&&B_1=
\left(
  \begin{array}{ccccc}
    \beta_1 & 0 & 0 & \ldots & 0 \\
    0 & \beta_2 & 0 & \ldots & 0 \\
    \ldots & \ldots & \ldots & \ldots & \ldots \\
    0 & 0 & 0 & 0 & \beta_m \\
  \end{array}
\right)
B=
\left(
  \begin{array}{ccccc}
    0 & \beta_1 & \beta_1 & \ldots & \beta_1 \\
    \beta_2 & 0 & \beta_2 & \ldots & \beta_2 \\
    \ldots & \ldots & \ldots & \ldots & \ldots \\
    \beta_m & \beta_m & \beta_m & \ldots & \beta_m \\
  \end{array}
\right)\\
&&\varphi(x):=x^TB_1x+\sum_{j=1}^N c_j(x_j).
\end{eqnarray*}
Note that when $c_j$ is differentiable convex for every $j$.\\

\noindent
We tested the proposed algorithm with the cost function given by
$$
c_j(x_j)=\frac{1}{2}x_j^TDx_j+d^Tx_j.
$$
\noindent
The parameters $\beta_j$ for all $j =1,\ldots,m$, matrix $D$ and vector $d$ were generated randomly
in the interval $(0,1]$, $[1,40]$ and $[1,40]$ respectively.\\

\noindent
We perform numerical implementations using different choices of $10$, and $20$, different initial choices $x_1$ generated randomly in the interval $[1,40]$ and $m=10$ with the stopping criterion as $\|x_{n+1}-x_n\|\leq10^{-2}$. Let us assume that each company have the same lower production bound 1 and upper production bound 40, that is,
$$K_i:=\{x_i: 1 \leq x_i \leq 40\},~~i=1,\ldots,10.$$
We compare our proposed Algorithm \ref{Alg:AlgL} with algorithm 3.2 proposed by Shehu and Iyiola in \cite{ShehuVIM}.\\

\begin{table}[H]
\caption{Example \ref{her1} Comparison: Proposed Alg. \ref{Alg:AlgL} and Shehu \& Iyiola Alg. 3.2 (SI Alg.) for $\sigma = 0.5$}
\centering 
\begin{tabular}{c c c c c c c c c c c c c c}
\toprule[1.5pt]
          &  \multicolumn{4}{c}{N=10} & & \multicolumn{4}{c}{N=20} \\
         \cline{2-5}\cline{7-10}\\
          &  \multicolumn{2}{c}{No. of Iter.} & \multicolumn{2}{c}{CPU time ($10^{-2}$)} & & \multicolumn{2}{c}{No. of Iter.} & \multicolumn{2}{c}{CPU time ($10^{-2}$)}\\
         \cline{2-5} \cline{7-10}
$\gamma$ & Alg. \ref{Alg:AlgL}  & SI Alg. & Alg. \ref{Alg:AlgL}  & SI Alg. & & Alg. \ref{Alg:AlgL}  & SI Alg. & Alg. \ref{Alg:AlgL}  & SI Alg.
\\
 \toprule[1.5pt] \\
$0.01$  & 223 & 435 & $2.6822$ & $8.8916$ && 228 & 520 & $7.9536$ & 22.352 \\ [0.5ex]
\hline \\
$0.1$  &  38 & 518 & $1.4433$ & 13.108 && 36 & 473  & $1.0797$ & 14.871 \\ [0.5ex]
 \hline \\
$0.5$  &  10 & 434 & $0.8232$ & $7.9301$ && 9 & 285 & $0.4196$ & $9.6452$ \\ [0.5ex]
 \hline\\
$0.8$  &  9 & 514 & 12.3590 & 13.1390 && 8 & 320 & $0.4596$ & $8.3524$ \\ [0.5ex]
 \hline
\end{tabular}
\label{table1}
\end{table}


\begin{figure}[H]
\minipage{0.50\textwidth}
\includegraphics[width=\linewidth]{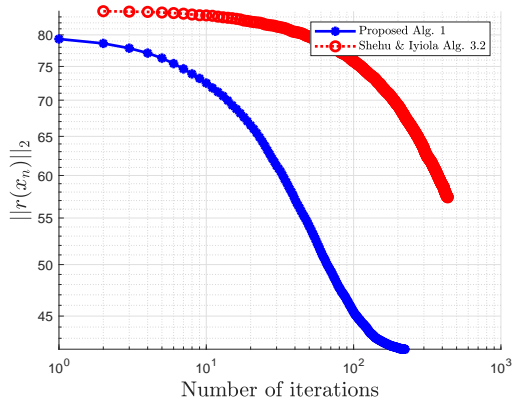}
\caption{Example 5.1: $\gamma = 0.01$, $N = 10$}\label{fig1}
\endminipage\hfill
\minipage{0.50\textwidth}
\includegraphics[width=\linewidth]{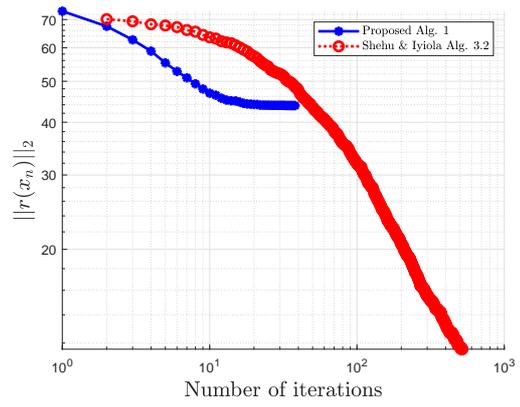}
\caption{Example 5.1: $\gamma = 0.1$, $N = 10$}\label{fig2}
\endminipage
\end{figure}

\begin{figure}[H]
\minipage{0.50\textwidth}
\includegraphics[width=\linewidth]{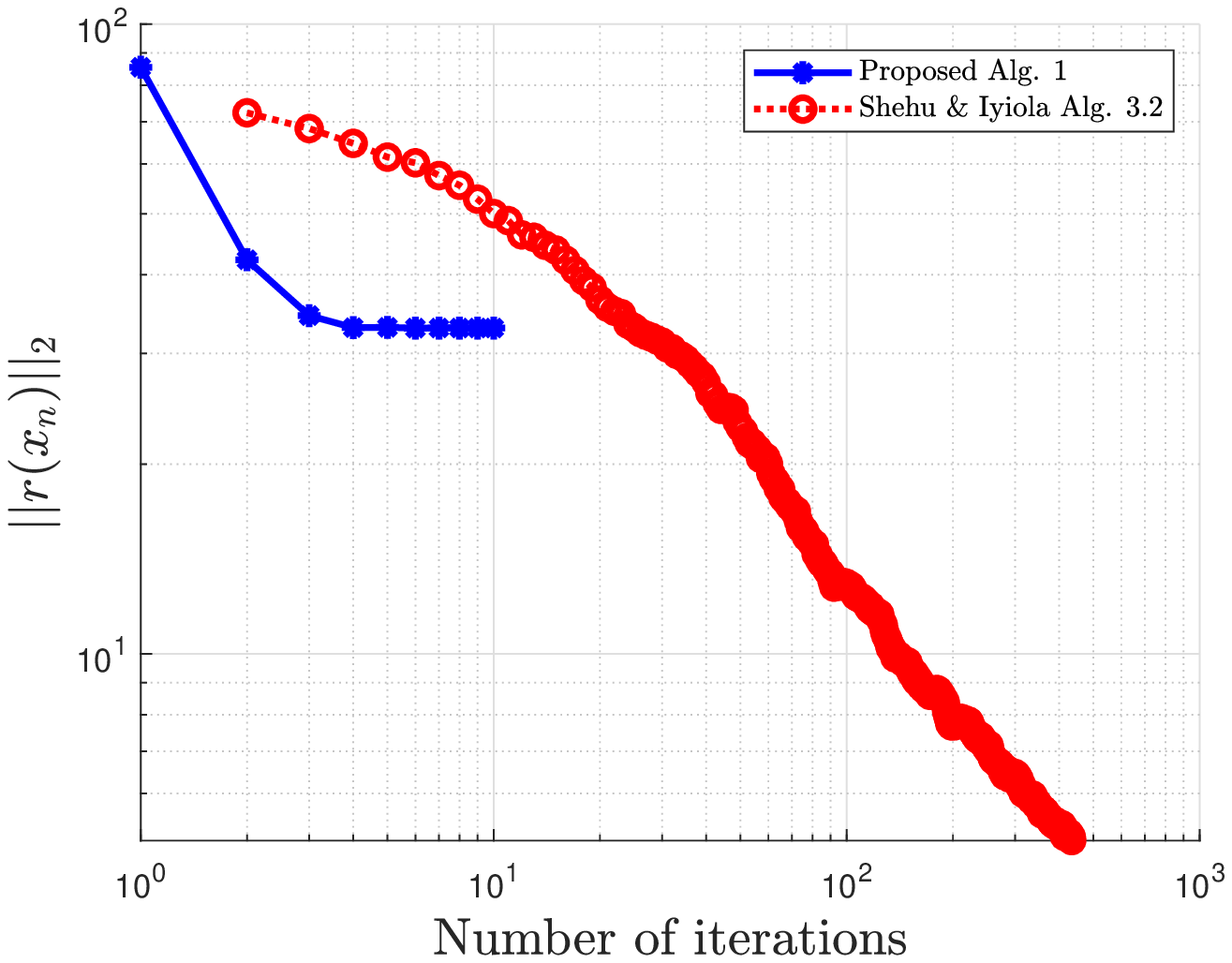}
\caption{Example 5.1: $\gamma = 0.5$, $N = 10$}\label{fig3}
\endminipage\hfill
\minipage{0.50\textwidth}
\includegraphics[width=\linewidth]{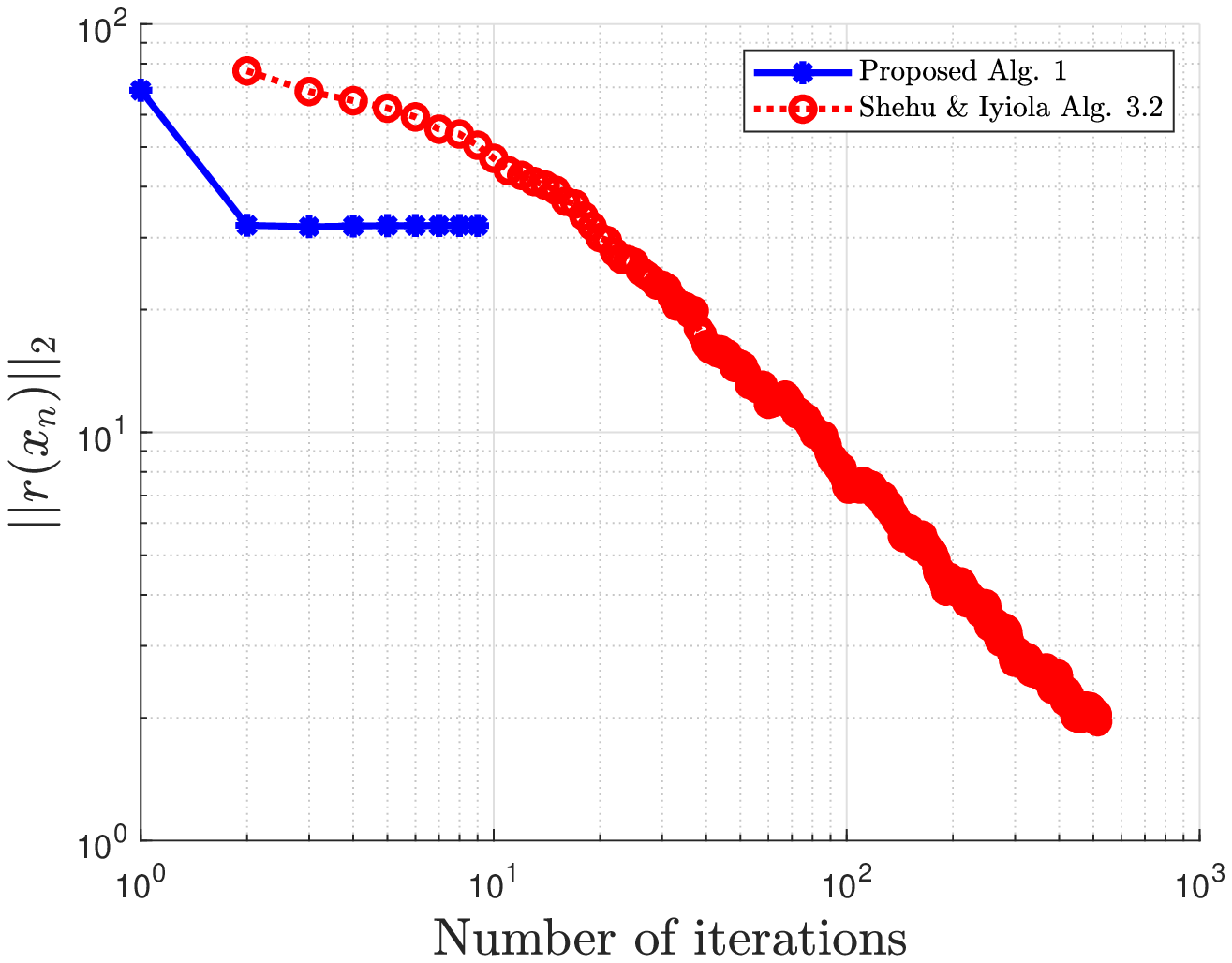}
\caption{Example 5.1: $\gamma = 0.7$, $N = 10$}\label{fig4}
\endminipage
\end{figure}

\begin{figure}[H]
\minipage{0.50\textwidth}
\includegraphics[width=\linewidth]{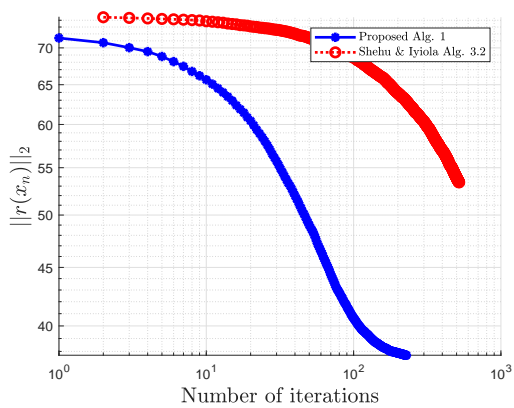}
\caption{Example 5.1: $\gamma = 0.01$, $N = 20$}\label{fig5}
\endminipage\hfill
\minipage{0.50\textwidth}
\includegraphics[width=\linewidth]{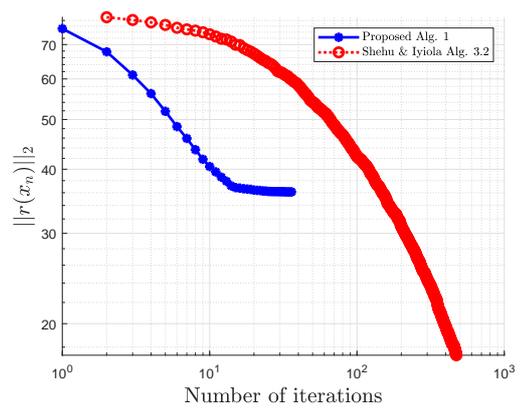}
\caption{Example 5.1: $\gamma = 0.1$, $N = 20$}\label{fig6}
\endminipage
\end{figure}

\begin{figure}[H]
\minipage{0.50\textwidth}
\includegraphics[width=\linewidth]{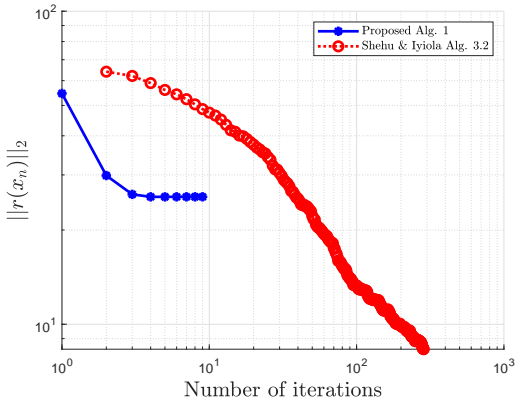}
\caption{Example 5.1: $\gamma = 0.5$, $N = 20$}\label{fig7}
\endminipage\hfill
\minipage{0.50\textwidth}
\includegraphics[width=\linewidth]{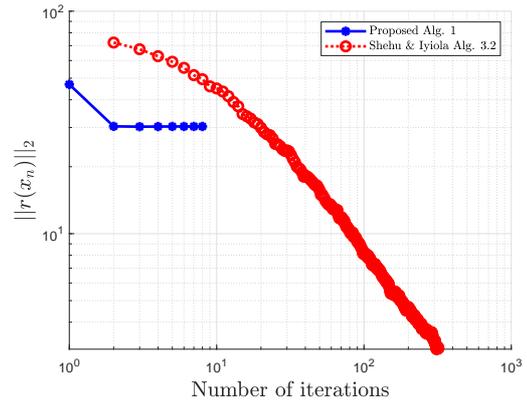}
\caption{Example 5.1: $\gamma = 0.7$, $N = 20$}\label{fig8}
\endminipage
\end{figure}
\end{exm}

\begin{exm}\label{her2}
This example is taken from \cite{Harker} and has been
considered by many authors for numerical experiments (see, for example,
\cite{Hieu, Malitsky, Solodov}). The operator $ A $ is defined by
$ A(x) := Mx+q $, where $ M= BB^T+ S + D $, with $ B, S, D \in
\mathbb R^{m \times m} $ randomly generated matrices such that $ S $ is
skew-symmetric (hence the operator does not arise from an optimization problem), $ D $ is a positive definite diagonal matrix (hence the variational inequality has a unique solution) and $ q = 0 $. The feasible set $ C $ is described by linear inequality constraints $ B x \leq b $ for some random matrix $B \in \mathbb R^{k \times m} $ and a random vector $ b \in \mathbb R^k $ with nonnegative entries. Hence the zero vector is feasible and therefore the unique solution of the corresponding variational inequality. These projections are computed using the MATLAB solver {\tt fmincon}. Hence, for this class of problems, the evaluation of $A$ is relatively inexpensive, whereas projections are costly. We
present the corresponding numerical results (number of iterations and
CPU times in seconds) using six different dimensions $m$ and two different numbers of inequality constraints $k$.\\

\noindent We choose the stopping criterion as $\|x^k\|\leq \epsilon = 0.001.$  The size $k =30, 50$ and $m =  10, 20, 30, 40, 50, 60$. The matrices $B,S,D$ and the vector $b$ are generated randomly. We choose  $\gamma=0.8$, $\sigma = 0.5$, $\alpha_n = 0.2$ in Algorithm \eqref{Alg:AlgL}. In \eqref{pppp3}, we choose $\sigma=0.8$, $\rho=0.1$, $\mu=0.2$. In \eqref{mm2}, we choose $L=\|M\|$. Here, we compare our proposed Algorithm \ref{Alg:AlgL} with the subgradient extragradient method (SEM) \eqref{pppp3}, and the inertial subgradient extragradient method (Thong \& Hieu) \eqref{mm2}. \\

\begin{sidewaystable}
\begin{table}[H]
\caption{Example \ref{her2} Comparison: Proposed Alg. \ref{Alg:AlgL} vs SEM (2) vs Thong \& Hieu (3) (T \& H (3))}
\centering 
\begin{tabular}{c c c c c c c c c c c c c}
\toprule[1.5pt]
         & &  \multicolumn{3}{c}{No. of Iterations} & & \multicolumn{3}{c}{CPU time} & & \multicolumn{3}{c}{Norm sol. ($10^{-3}$)}\\
         \cline{3-5} \cline{7-9} \cline{11-13}
  & $m$ & \textbf{Alg. 1} & \textbf{SEM (2)} & \textbf{T \& H (3)} & &\textbf{Alg. 1} & \textbf{SEM (2)}&\textbf{T \& H (3)} & &\textbf{Alg. 1} & \textbf{SEM (2)} & \textbf{T \& H (3)}\\
\toprule[1.5pt]
 $k=30$ & \textbf{10} & 344 & 3867 & 4123 & & 2.8707 & 31.5956 & 30.3300 & & 0.99605 & 0.99927 & 0.99938\\
\cline{2-13}
& \textbf{20} & 747 & 14683 & 10493 & & 5.6957 & 117.7458 & 84.7406 & & 0.99608 & 0.99996 & 0.99984\\
\cline{2-13}
& \textbf{30} & 1777 & 31668 & 24968 & & 13.891 & 269.2987 & 211.1937 & & 0.99955 & 0.99999 & 0.99980\\
\cline{2-13}
& \textbf{40} & 2612 & 40224 & 36119 & & 21.5972 & 358.3453 & 320.2933 & & 0.99790 & 1.00000 & 0.99994\\
\cline{2-13}
& \textbf{50} & 3710 & 70321 & 51143 & & 32.074 & 655.8354 & 469.0297 & & 0.99981 & 0.99997 & 0.99995\\
\cline{2-13}
 & \textbf{60} & 5619 & 56670 & 50619 & & 50.4537 & 554.3951 & 491.5552 &  & 0.99929 & 0.99992 & 0.99998\\
\toprule[1.5pt]
$k=50$ & \textbf{10} & 200 & 6213 & 5518 & & 1.90869 & 60.1212 & 47.6937 & & 0.98471 & 0.99969 & 0.99945\\
\cline{2-13}
& \textbf{20} & 835 & 14354 & 10372 & & 6.4942 & 126.429 & 96.9197 & & 0.99909 & 0.99980 & 0.99991\\
\cline{2-13}
& \textbf{30} & 1978 & 25519 & 19357 & & 16.4674 & 240.61 & 208.910 & & 0.99794 & 0.99991 & 0.99990\\
\cline{2-13}
& \textbf{40} & 2832 & 47661 & 26790 & & 30.3799 & 539.729 & 314.588 & & 0.99734 & 0.99991 & 0.99938 \\
\cline{2-13}
& \textbf{50} & 3933 & 43773 & 53055 & & 45.8745 & 562.959 & 800.371 & & 0.99985 & 0.99925 & 0.99999\\
\cline{2-13}
& \textbf{60} & 6025 & 97772 & 65820 & & 100.304 & 1515.76 & 589.180 & & 0.99955 & 0.99995 & 0.99994\\
\toprule[1.5pt]

\end{tabular}\label{table2}
\end{table}
\end{sidewaystable}

\begin{figure}[H]
\minipage{0.33\textwidth}
\includegraphics[width=\linewidth]{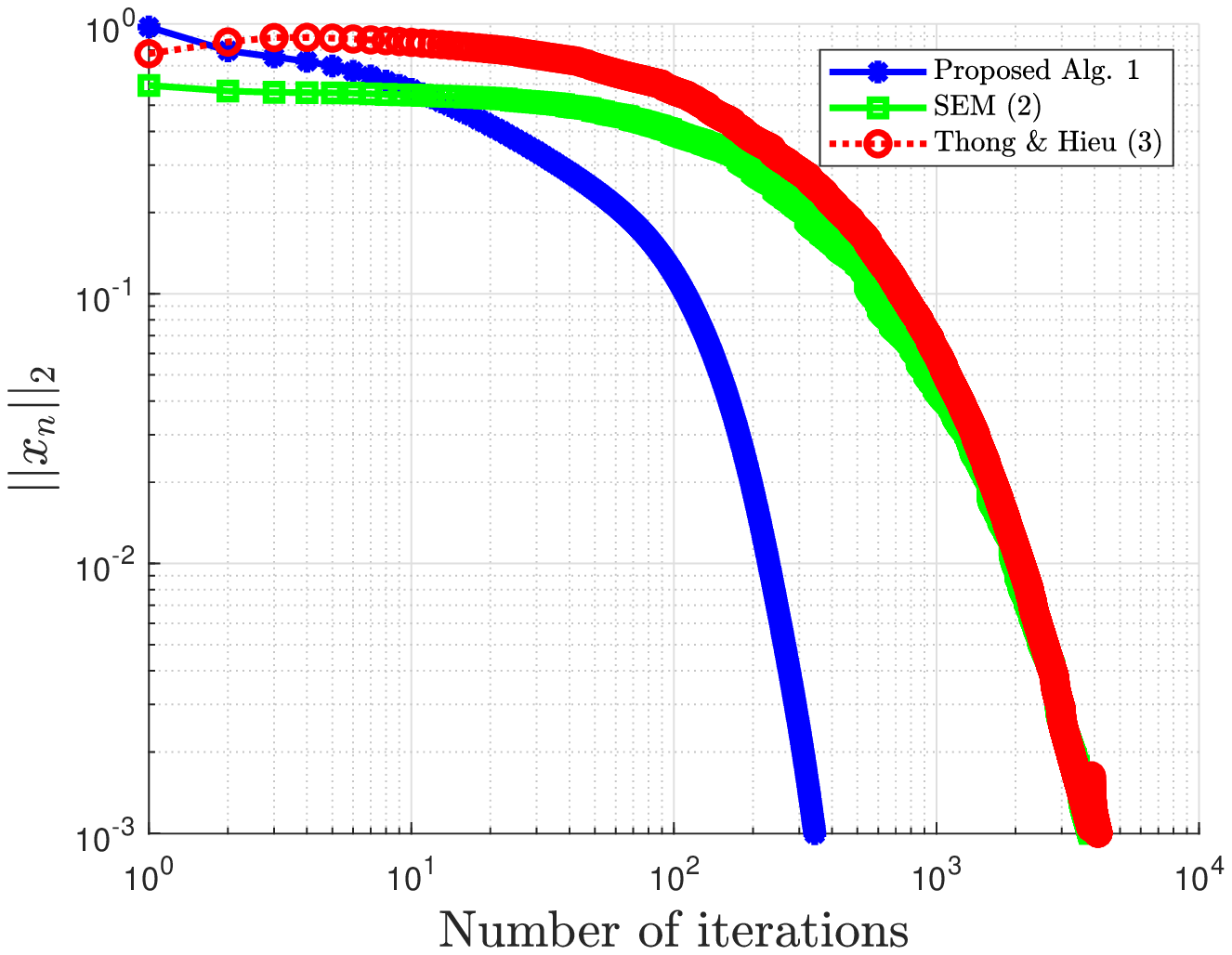}
\caption{Example \ref{her2}: $k = 30$, $m = 10$}\label{fig9}
\endminipage\hfill
\minipage{0.33\textwidth}
\includegraphics[width=\linewidth]{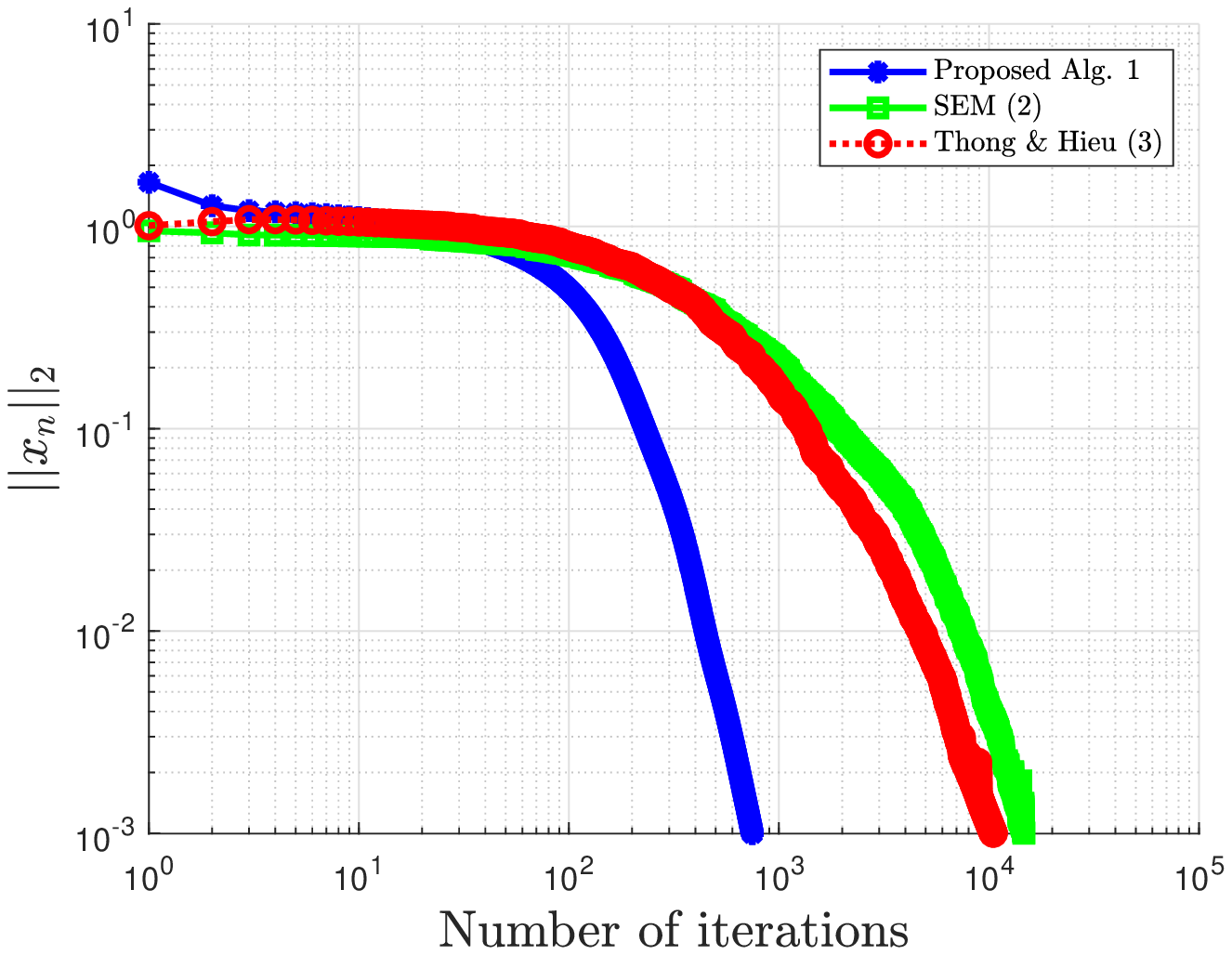}
\caption{Example \ref{her2}: $k = 30$, $m = 20$}\label{fig10}
\endminipage\hfill
\minipage{0.33\textwidth}
\includegraphics[width=\linewidth]{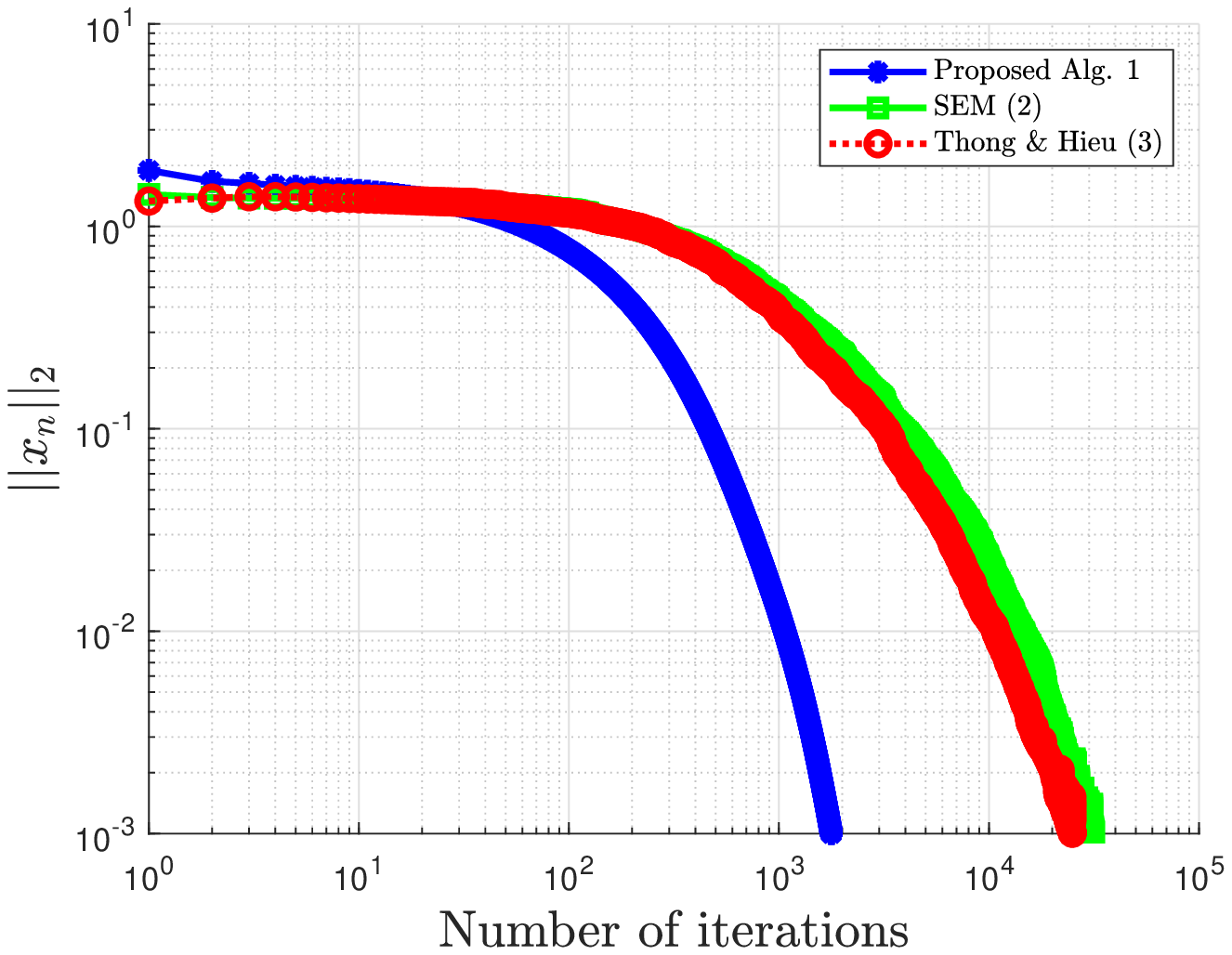}
\caption{Example \ref{her2}: $k = 30$, $m = 30$}\label{fig11}
\endminipage
\end{figure}

\begin{figure}[H]
\minipage{0.33\textwidth}
\includegraphics[width=\linewidth]{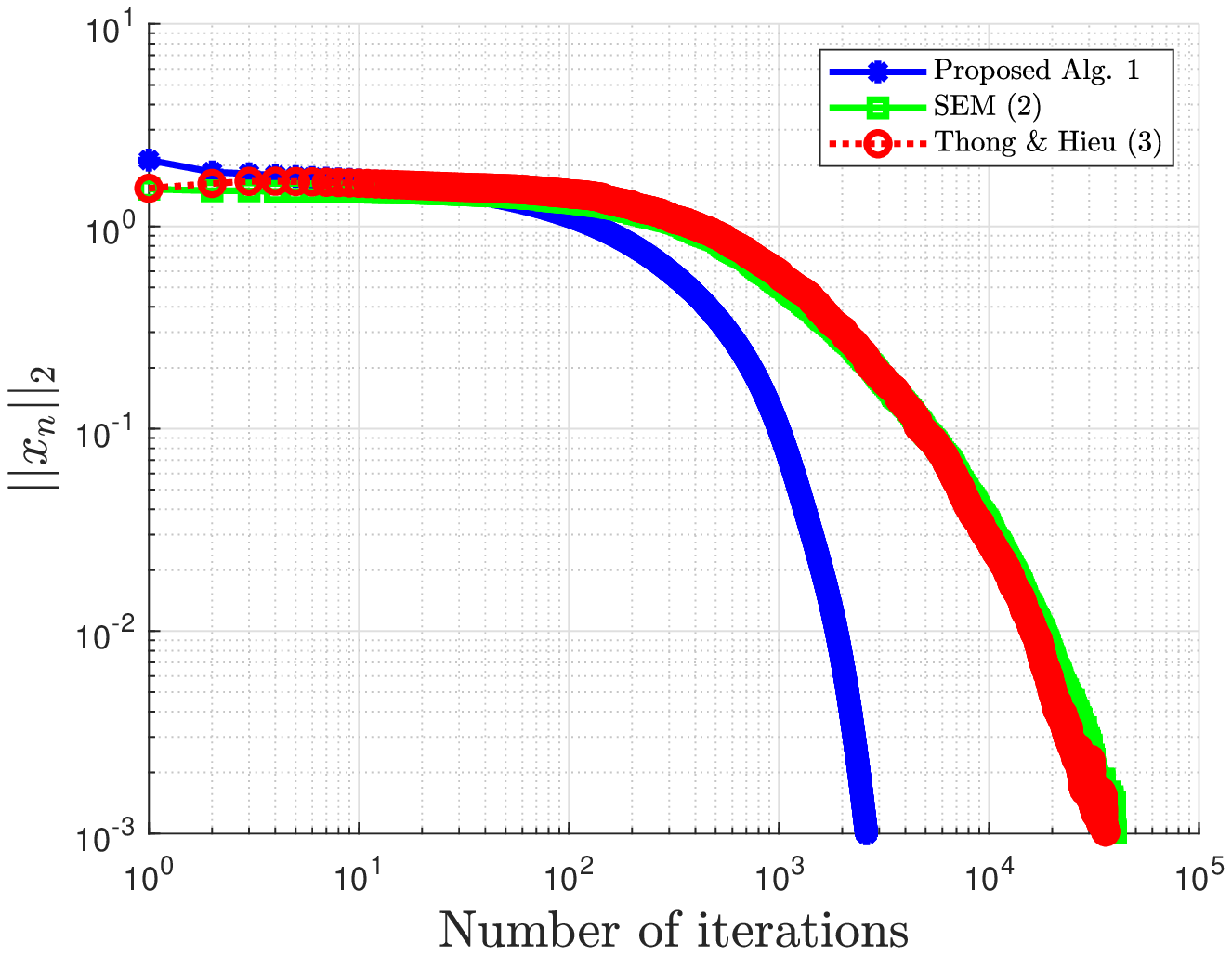}
\caption{Example \ref{her2}: $k = 30$, $m = 40$}\label{fig12}
\endminipage\hfill
\minipage{0.33\textwidth}
\includegraphics[width=\linewidth]{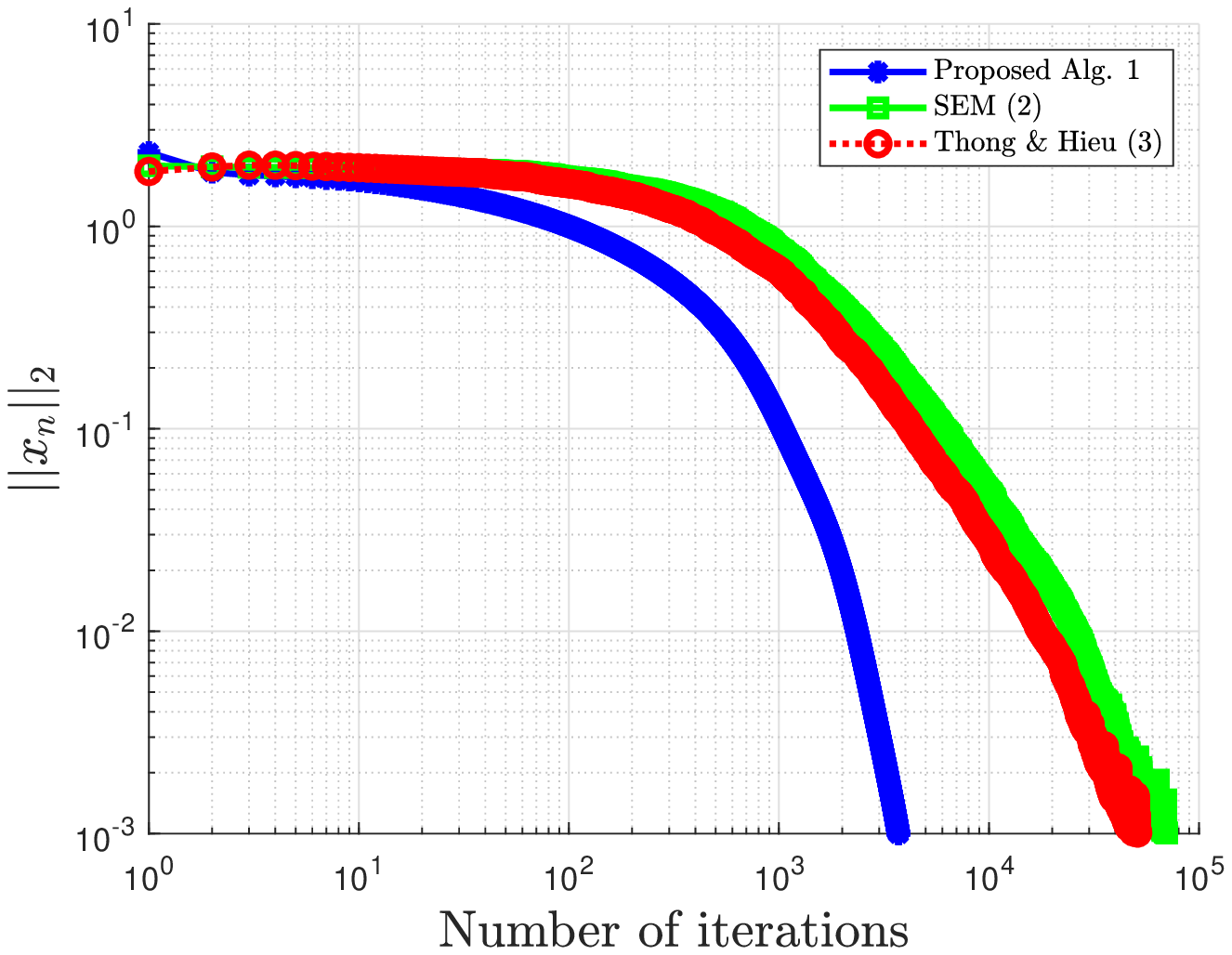}
\caption{Example \ref{her2}: $k = 30$, $m = 50$}\label{fig13}
\endminipage\hfill
\minipage{0.33\textwidth}
\includegraphics[width=\linewidth]{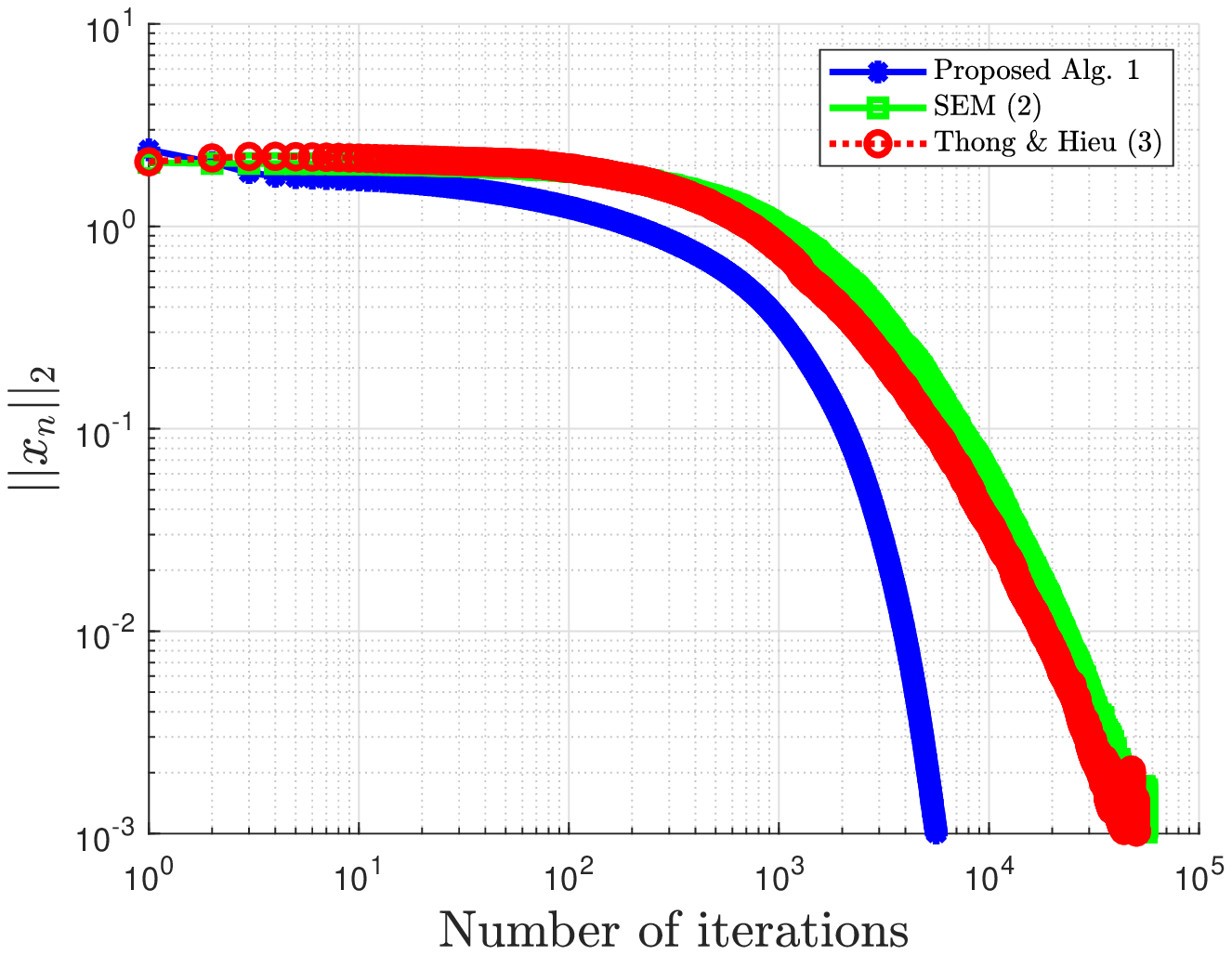}
\caption{Example \ref{her2}: $k = 30$, $m = 60$}\label{fig14}
\endminipage
\end{figure}

\begin{figure}[H]
\minipage{0.33\textwidth}
\includegraphics[width=\linewidth]{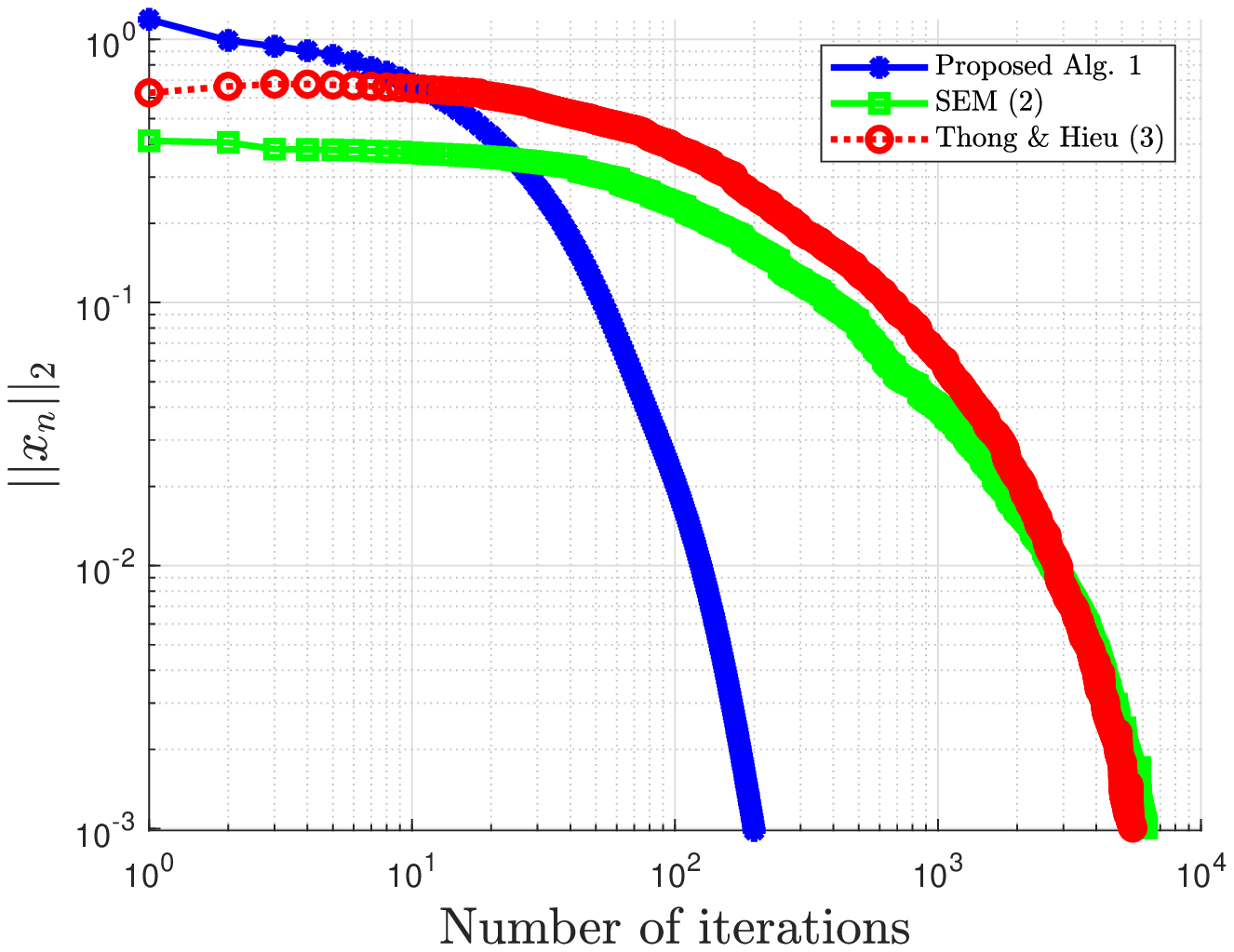}
\caption{Example \ref{her2}: $k = 50$, $m = 10$}\label{fig15}
\endminipage\hfill
\minipage{0.33\textwidth}
\includegraphics[width=\linewidth]{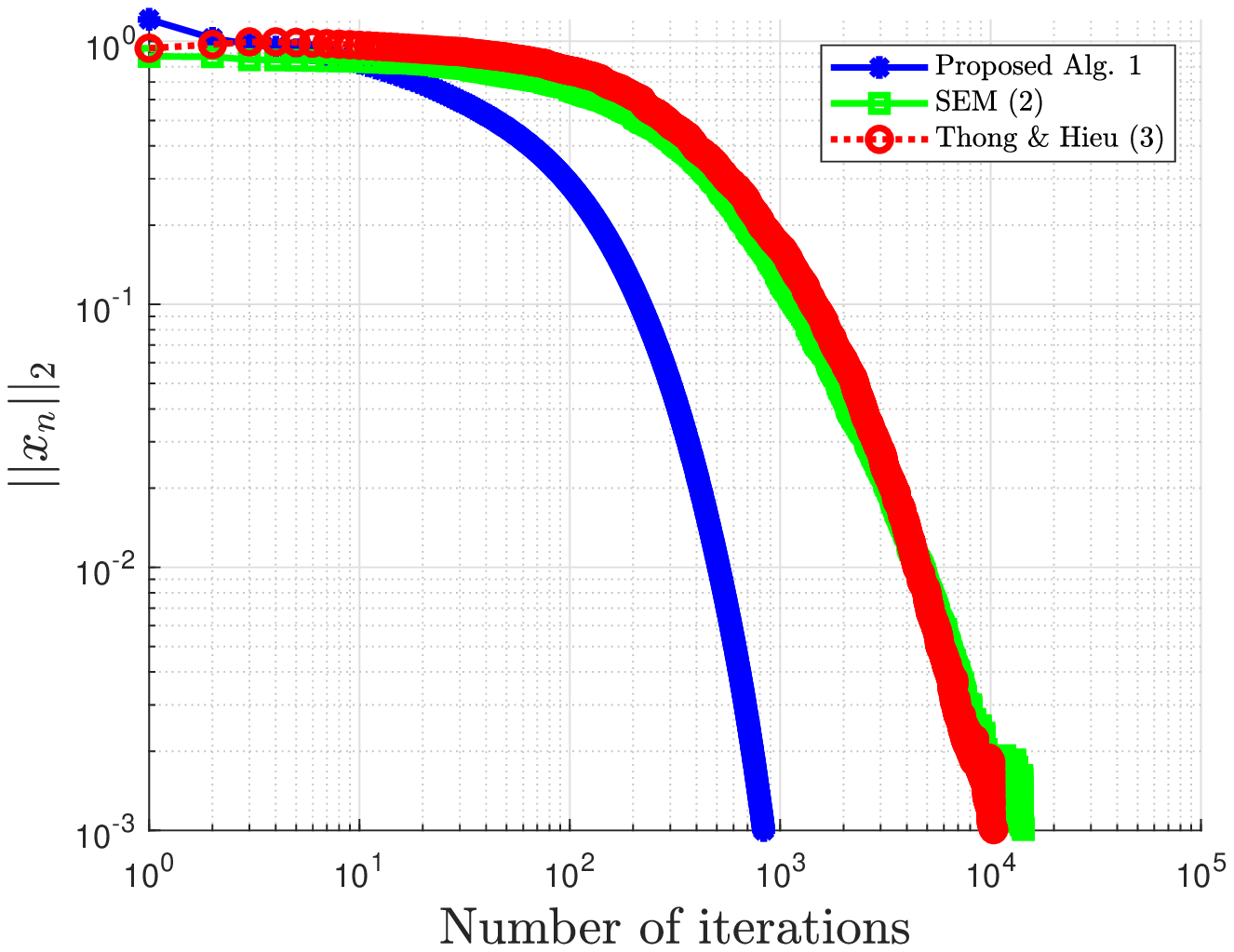}
\caption{Example \ref{her2}: $k = 50$, $m = 20$}\label{fig16}
\endminipage\hfill
\minipage{0.33\textwidth}
\includegraphics[width=\linewidth]{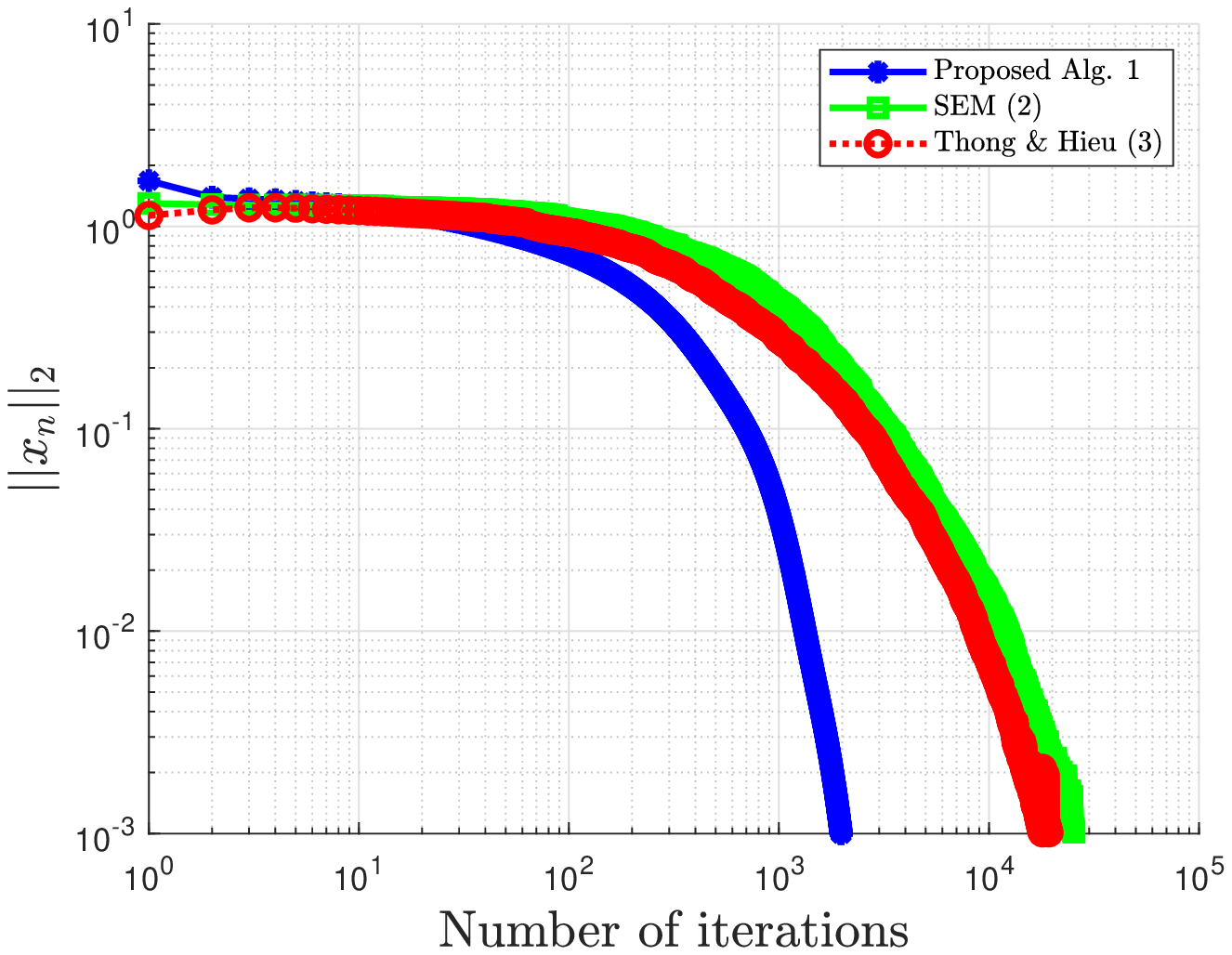}
\caption{Example \ref{her2}: $k = 50$, $m = 30$}\label{fig17}
\endminipage
\end{figure}

\begin{figure}[H]
\minipage{0.33\textwidth}
\includegraphics[width=\linewidth]{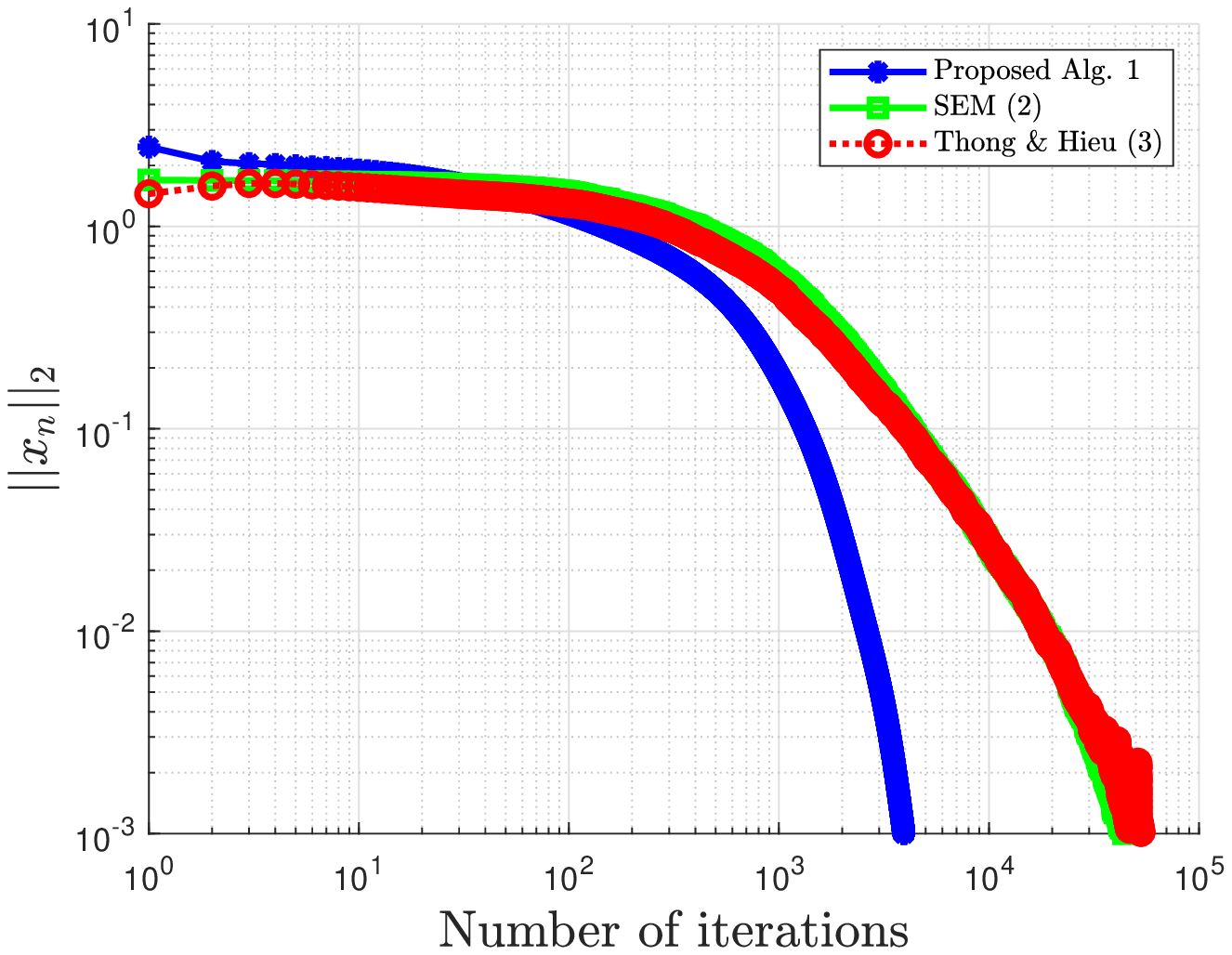}
\caption{Example \ref{her2}: $k = 50$, $m = 50$}\label{fig19}
\endminipage\hfill
\minipage{0.33\textwidth}
\includegraphics[width=\linewidth]{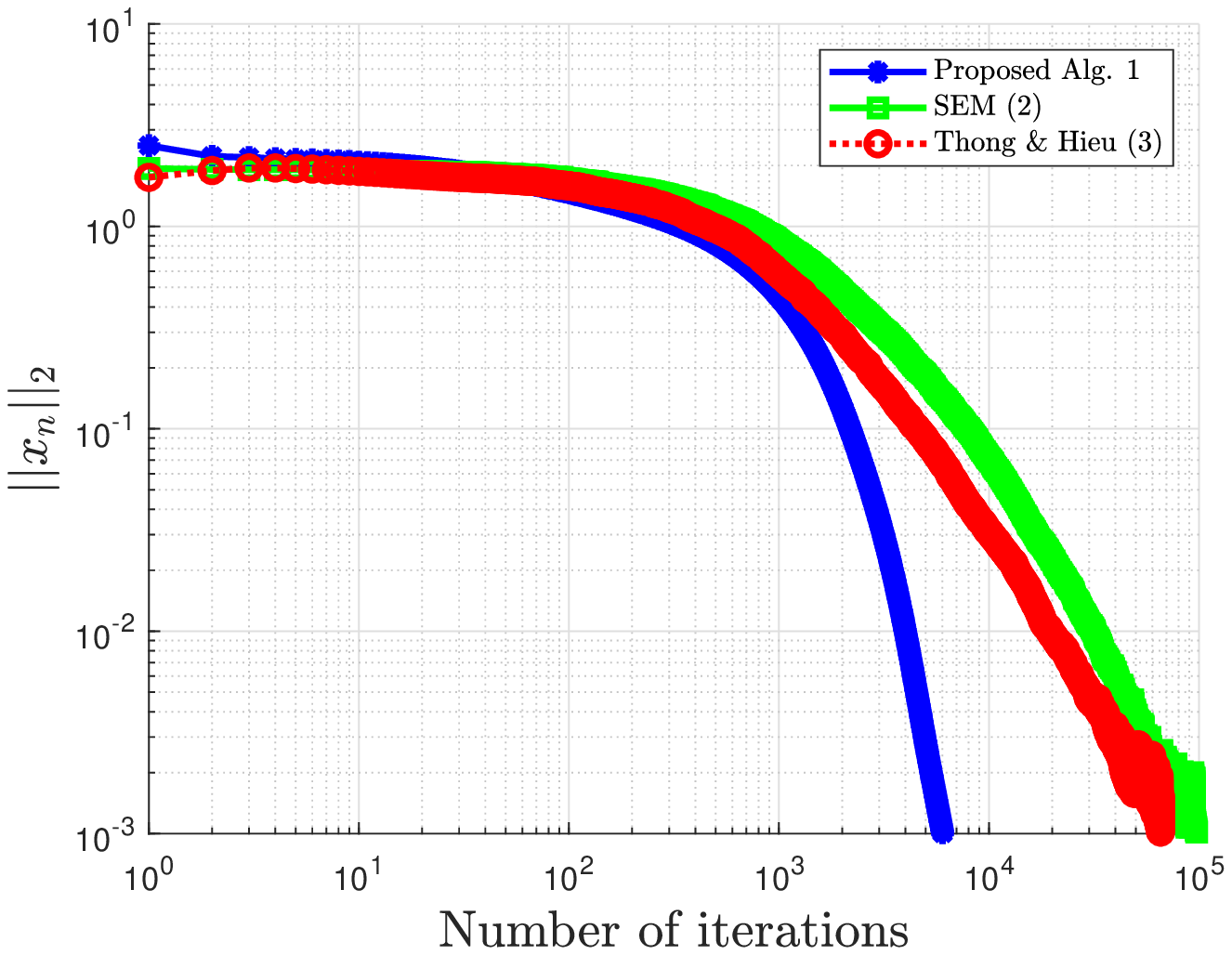}
\caption{Example \ref{her2}: $k = 50$, $m = 60$}\label{fig20}
\endminipage
\end{figure}


Clearly, from both Examples, our proposed algorithm \ref{Alg:AlgL} outperforms and highly improves Shehu and Iyiola Algorithm (3.2) in \cite{ShehuVIM}, subgradient extragradient method (SEM) \eqref{pppp3}, and the inertial subgradient extragradient method (Thong \& Hieu) \eqref{mm2} with respect to number of iterations required and CPU time and achieved norm of
the solution. See Tables \ref{table1} - \ref{table2} and Figures \ref{fig1} - \ref{fig20}.

\end{exm}

\noindent
We give an example in infinite dimensional Hilbert spaces. We give comparison of our proposed Algorithm \ref{Alg:AlgL} with Algorithm~\eqref{pppp3}, Algorithm~\eqref{mm2} and the non-inertial case of Algorithm \ref{Alg:AlgL} (when $\theta_n=0$).

\begin{exm}\label{adede}
Let $H:=L^2([0,1])$ with norm $\|x\|:=\Big( \int_{0}^{1}x(t)^2dt\Big)^{\frac{1}{2}}$ and inner product
$\langle x,y \rangle:=\int_{0}^{1}x(t)y(t)dt,~~x,y \in H$. Let $C:=\{x \in L^2([0,1]):\int_{0}^{1}tx(t)dt= 2\}$. Let us define the Volterra integral operator $F:L^2([0,1])\rightarrow L^2([0,1])$ by $Fx(t):=\int_{0}^{t}x(s)ds,~~x \in L^2([0,1]), t \in [0,1]$. Then, $F$ is monotone, bounded and linear with $L=\frac{2}{\pi}$ (see Exercises 20.12 of \cite{Bauschkebook}) . Observe that $\text{SOL}\neq \emptyset$ since $0 \in \text{SOL}$. Observe that (see \cite{Cegielskibook})
$$
P_C(x)(t):=x(t)-\frac{\int_{0}^{1}tx(t)dt-2}{\int_{0}^{1}t^{2}dt}t, ~~t \in [0,1].
$$
\end{exm}

\section{Final Remarks}\label{Sec:Final}
We propose an inertial projection method for solving variational inequality problem and give weak convergence result. The cost function is assumed to be monotone and non-Lipschitz continuous. Our numerical implementations show that our method is more efficient and outperforms some other related methods in the literature. Our result is more applicable than the results on variational inequality where the Lipschitz constant of the cost function is needed. Our future project is focused on how to extend the range of inertial factor $\alpha_n$ beyond $1/3$ and extend our results to infinite dimensional Banach spaces.

\section*{Acknowledgments}
The authors are grateful to the anonymous referee and editor whose insightful comments and suggestions improve the earlier version of this paper. 

\section*{Disclosure statement}
No potential conflict of interest was reported by the author(s).

\section*{Funding}
The project of the first author has received funding from the European Research Council (ERC) under the European Union’s Seventh Framework Program (FP7 - 2007-2013) (Grant agreement No. 616160).

\end{document}